\documentclass[12pt]{article}
\usepackage{amsmath}
\usepackage{amssymb,latexsym}
\usepackage{amsthm}
\usepackage{bm}
\newdimen\proofrulebreadth \proofrulebreadth=.05em
\newdimen\proofdotseparation \proofdotseparation=1.25ex
\newdimen\proofrulebaseline \proofrulebaseline=2ex
\newcount\proofdotnumber \proofdotnumber=3
\let\then\relax
\def\hfi{\hskip0pt plus.0001fil}
\mathchardef\squigto="3A3B
\newif\ifinsideprooftree\insideprooftreefalse
\newif\ifonleftofproofrule\onleftofproofrulefalse
\newif\ifproofdots\proofdotsfalse
\newif\ifdoubleproof\doubleprooffalse
\let\wereinproofbit\relax
\newdimen\shortenproofleft
\newdimen\shortenproofright
\newdimen\proofbelowshift
\newbox\proofabove
\newbox\proofbelow
\newbox\proofrulename
\def\shiftproofbelow{\let\next\relax\afterassignment\setshiftproofbelow\dimen0 }
\def\shiftproofbelowneg{\def\next{\multiply\dimen0 by-1 }%
\afterassignment\setshiftproofbelow\dimen0 }
\def\setshiftproofbelow{\next\proofbelowshift=\dimen0 }
\def\setproofrulebreadth{\proofrulebreadth}

\def\prooftree{
\ifnum  \lastpenalty=1
\then   \unpenalty
\else   \onleftofproofrulefalse
\fi
\ifonleftofproofrule
\else   \ifinsideprooftree
        \then   \hskip.5em plus1fil
        \fi
\fi
\bgroup
\setbox\proofbelow=\hbox{}\setbox\proofrulename=\hbox{}%
\let\justifies\proofover\let\leadsto\proofoverdots\let\Justifies\proofoverdbl
\let\using\proofusing\let\[\prooftree
\ifinsideprooftree\let\]\endprooftree\fi
\proofdotsfalse\doubleprooffalse
\let\thickness\setproofrulebreadth
\let\shiftright\shiftproofbelow \let\shift\shiftproofbelow
\let\shiftleft\shiftproofbelowneg
\let\ifwasinsideprooftree\ifinsideprooftree
\insideprooftreetrue
\setbox\proofabove=\hbox\bgroup$\displaystyle 
\let\wereinproofbit\prooftree
\shortenproofleft=0pt \shortenproofright=0pt \proofbelowshift=0pt
\onleftofproofruletrue\penalty1
}

\def\eproofbit{
\ifx    \wereinproofbit\prooftree
\then   \ifcase \lastpenalty
        \then   \shortenproofright=0pt  
        \or     \unpenalty\hfil         
        \or     \unpenalty\unskip       
        \else   \shortenproofright=0pt  
        \fi
\fi
\global\dimen0=\shortenproofleft
\global\dimen1=\shortenproofright
\global\dimen2=\proofrulebreadth
\global\dimen3=\proofbelowshift
\global\dimen4=\proofdotseparation
\global\count255=\proofdotnumber
$\egroup  
\shortenproofleft=\dimen0
\shortenproofright=\dimen1
\proofrulebreadth=\dimen2
\proofbelowshift=\dimen3
\proofdotseparation=\dimen4
\proofdotnumber=\count255
}

\def\proofover{
\eproofbit 
\setbox\proofbelow=\hbox\bgroup 
\let\wereinproofbit\proofover
$\displaystyle
}%
\def\proofoverdbl{
\eproofbit 
\doubleprooftrue
\setbox\proofbelow=\hbox\bgroup 
\let\wereinproofbit\proofoverdbl
$\displaystyle
}%
\def\proofoverdots{
\eproofbit 
\proofdotstrue
\setbox\proofbelow=\hbox\bgroup 
\let\wereinproofbit\proofoverdots
$\displaystyle
}%
\def\proofusing{
\eproofbit 
\setbox\proofrulename=\hbox\bgroup 
\let\wereinproofbit\proofusing
\kern0.3em$
}

\def\endprooftree{
\eproofbit 
  \dimen5 =0pt
\dimen0=\wd\proofabove \advance\dimen0-\shortenproofleft
\advance\dimen0-\shortenproofright
\dimen1=.5\dimen0 \advance\dimen1-.5\wd\proofbelow
\dimen4=\dimen1
\advance\dimen1\proofbelowshift \advance\dimen4-\proofbelowshift
\ifdim  \dimen1<0pt
\then   \advance\shortenproofleft\dimen1
        \advance\dimen0-\dimen1
        \dimen1=0pt
        \ifdim  \shortenproofleft<0pt
        \then   \setbox\proofabove=\hbox{%
                        \kern-\shortenproofleft\unhbox\proofabove}%
                \shortenproofleft=0pt
        \fi
\fi
\ifdim  \dimen4<0pt
\then   \advance\shortenproofright\dimen4
        \advance\dimen0-\dimen4
        \dimen4=0pt
\fi
\ifdim  \shortenproofright<\wd\proofrulename
\then   \shortenproofright=\wd\proofrulename
\fi
\dimen2=\shortenproofleft \advance\dimen2 by\dimen1
\dimen3=\shortenproofright\advance\dimen3 by\dimen4
\ifproofdots
\then
        \dimen6=\shortenproofleft \advance\dimen6 .5\dimen0
        \setbox1=\vbox to\proofdotseparation{\vss\hbox{$\cdot$}\vss}%
        \setbox0=\hbox{%
                \advance\dimen6-.5\wd1
                \kern\dimen6
                $\vcenter to\proofdotnumber\proofdotseparation
                        {\leaders\box1\vfill}$%
                \unhbox\proofrulename}%
\else   \dimen6=\fontdimen22\the\textfont2 
        \dimen7=\dimen6
        \advance\dimen6by.5\proofrulebreadth
        \advance\dimen7by-.5\proofrulebreadth
        \setbox0=\hbox{%
                \kern\shortenproofleft
                \ifdoubleproof
                \then   \hbox to\dimen0{%
                        $\mathsurround0pt\mathord=\mkern-6mu%
                        \cleaders\hbox{$\mkern-2mu=\mkern-2mu$}\hfill
                        \mkern-6mu\mathord=$}%
                \else   \vrule height\dimen6 depth-\dimen7 width\dimen0
                \fi
                \unhbox\proofrulename}%
        \ht0=\dimen6 \dp0=-\dimen7
\fi
\let\doll\relax
\ifwasinsideprooftree
\then   \let\VBOX\vbox
\else   \ifmmode\else$\let\doll=$\fi
        \let\VBOX\vcenter
\fi
\VBOX   {\baselineskip\proofrulebaseline \lineskip.2ex
        \expandafter\lineskiplimit\ifproofdots0ex\else-0.6ex\fi
        \hbox   spread\dimen5   {\hfi\unhbox\proofabove\hfi}%
        \hbox{\box0}%
        \hbox   {\kern\dimen2 \box\proofbelow}}\doll%
\global\dimen2=\dimen2
\global\dimen3=\dimen3
\egroup 
\ifonleftofproofrule
\then   \shortenproofleft=\dimen2
\fi
\shortenproofright=\dimen3
\onleftofproofrulefalse
\ifinsideprooftree
\then   \hskip.5em plus 1fil \penalty2
\fi
}

\usepackage[all,cmtip]{xy}
\input{diagxy}
\usepackage{url}
\newcommand{\cors}[1]{\ensuremath{[#1]}}

\newcommand{\C}{\ensuremath{\mathbb{C}}}
\newcommand{\Chat}{\ensuremath{\widehat{\mathbb{C}}}}
\newcommand{\D}{\ensuremath{\mathcal{D}}}
\newcommand{\CC}{\ensuremath{\mathcal{E}}}

\newcommand{\Set}{\ensuremath{\mathsf{Set}}}
\newcommand{\Cat}{\ensuremath{\mathsf{Cat}}}

\newcommand{\op}{\ensuremath{^\mathrm{op}}}

\newcommand{\pbcorner}[1][dr]{\save*!/#1-1.2pc/#1:(-1,1)@^{|-}\restore}
\newcommand{\dom}{\ensuremath{\mathsf{dom}}}
\newcommand{\cod}{\ensuremath{\mathsf{cod}}}
\newcommand{\identity}{\ensuremath{\mathsf{id}}}
\newcommand{\ev}{\ensuremath{\mathrm{ev}}}
\newcommand{\y}{\ensuremath{\mathsf{y}}}

\newcommand{\hook}{\ensuremath{\hookrightarrow}}

\newcommand{\G}{\ensuremath{\Gamma}}

\newcommand{\type}{\mathsf{type}}       
\newcommand{\types}[2]{#1 \vdash #2:\type}
\newcommand{\Gtypes}[1]{\types{\Gamma}{#1}}
\newcommand{\term}[2]{#1\,:\,#2}
\newcommand{\terms}[2]{#1 \vdash #2}
\newcommand{\Gterms}[1]{\terms{\Gamma}{#1}}
\newcommand{\ext}[2]{{#1\!\centerdot\! #2}}
\newcommand{\ty}{\ensuremath{\,:\,}}
\newcommand{\pair}[1]{\ensuremath{\langle #1\rangle}}

\newcommand{\pairmap}{\ensuremath{\mathsf{pair}}}
\newcommand{\exdot}{\ensuremath{\!\centerdot\!}}

\newcommand{\btexdot}[2]{\ensuremath{{#1}\!\centerdot\!{#2}}}
\newcommand{\ttexdot}[3]{\ensuremath{{#1}\!\centerdot\!{#2}\!\centerdot\!{#3}}}
\newcommand{\qtexdot}[4]{\ensuremath{{#1}\!\centerdot\!{#2}\!\centerdot\!{#3}}\!\centerdot\!{#4}}

\newcommand{\Id}{\mathsf{Id}}
\newcommand{\id}[1]{\Id_{#1}}

\newcommand{\jay}{\mathsf{j}}
\newcommand{\iy}{\mathsf{i}}

\newcommand{\U}{\ensuremath{\mathcal{U}}}
\newcommand{\UU}{\ensuremath{\widetilde{\mathcal{U}}}}

\newtheorem{theorem}{Theorem}
\newtheorem*{theorem*}{Theorem}
\newtheorem{proposition}[theorem]{Proposition} 
\newtheorem{lemma}[theorem]{Lemma}
\newtheorem{corollary}[theorem]{Corollary} 

\theoremstyle{definition}
\newtheorem{definition}[theorem]{Definition}
\newtheorem{remark}[theorem]{Remark} 
\newtheorem*{remarks*}{Remarks}

\begin{document}

\title{Natural models of homotopy type theory\thanks{
Penultimate version; published as \cite{awodeyNatural}}}
\author{Steve Awodey}
\date{\today}

\maketitle

\begin{abstract}
\noindent The notion of a \emph{natural model} of type theory is defined in terms of that of a \emph{representable natural transfomation} of presheaves.  It is shown that such models agree exactly with the concept of a \emph{category with families} in the sense of Dybjer, which can be regarded as an algebraic formulation of type theory.  We determine conditions for such models to satisfy the inference rules for dependent sums $\Sigma$, dependent products $\Pi$, and intensional identity types $\Id$, as used in homotopy type theory.  It is then shown that a category admits such a model if it has a class of maps that behave like the abstract fibrations in axiomatic homotopy theory:\ they should be stable under pullback, closed under composition and relative products, and there should be weakly orthogonal factorizations into the class.  It follows that many familiar settings for homotopy theory also admit natural models of the basic system of homotopy type theory.
\end{abstract}


\noindent Homotopy type theory is an interpretation of constructive Martin-L\"of type theory \cite{ML}  into abstract homotopy theory.  It allows type theory to be used as a formal calculus for reasoning about homotopy theory, as well as more general mathematics such as can be formulated in category theory or set theory under this new interpretation.  Because constructive type theory has been implemented in computational proof assistants like \textsc{Coq}, homotopy type theory also facilitates the use of such computational tools in homotopy theory, category theory, set theory, and other fields of mathematics.  This is just one aspect of the \emph{Univalent Foundations Program}, which has recently been the object of quite intense investigation \cite{HoTTbook}.

One thing missing from homotopy type theory, however, has been a notion of \emph{model} that is both faithful to the precise formalism of type theory and yet general and flexible enough to be a practical tool for semantic investigations.  Past attempts have involved either highly structured categories corresponding closely to the syntax of type theory, such as the \emph{categories with families} of Dybjer \cite{CwF}, which are, however, somewhat impractical to work with semantically; or they use the more more natural and flexible setting of homotopical algebra, as in \cite{AW,GvdB}, but they must then be equipped (if possible) with structures satisfying unnatural coherence conditions, in order to model the type theory precisely.

Here we present a new approach which attemps to combine the advantages of these two strategies. It is based on the observation that a category with families is the same thing as a representable natural transformation in the sense of Grothendieck.  Ideas from Voevodsky \cite{KLV} and Lumsdaine-Warren \cite{LW} are also used in an essential way.  In the first section, the basic concept of a natural model is defined and shown to be adequate.  The second section determines conditions for when the basic type constructors $\Sigma, \Pi, \Id$ are also modelled.  This draws heavily on the methodology of \cite{KLV}.  Finally, the third section investigates the question of when a category admits such a model, concluding with the main result which provides a general, sufficient condition. This is closely related to the main result of \cite{LW}, which uses similar reasoning.

\section{Natural models}

The following concept is usually attributed to Grothendieck and is widely used in the theory of stacks (cf.~Def.\ 4.8.2~\cite[Tag 0023]{stacks-project}).
\begin{definition}\label{def:rep}
Let \C\ be a small category.  A natural transformation $$f : Y \to X$$ of presheaves on \C\ is called \emph{representable} if all of its fibers are representable objects, in the following sense: for every $C\in\C$ and $x\in X(C)$, there is a $D\in\C$, a $p : D\to C$, and a $y\in Y(D)$ such that the following square is a pullback,
\begin{equation}\label{diag:rep}
\xymatrix{
 \y{D} \ar[d]_{\y{p}} \ar[r]^-{y} \pbcorner &  Y\ar[d]^{f}\\
\y{C} \ar[r]_{x}   & X .}
\end{equation}
As here, we shall freely use the Yoneda lemma to identify elements $x\in X(C)$ with natural maps $x:\y{C} \to X$.
\end{definition}

Our first observation is that a representable natural transformation is the same thing as a \emph{category with families} in the sense of Dybjer \cite{CwF}.  
Indeed, let us write the objects of \C\ as $\Gamma, \Delta, \ldots$ and the arrows as $\sigma : \Delta \to \Gamma, \dots$, thinking of \C\ as a ``category of contexts".  Let $p : \UU\to\U$ be a representable map of presheaves, and write its elements as:
\begin{align*}
A\in \U(\Gamma)\ &\Leftrightarrow\ \Gtypes{A}\\
a\in \UU(\Gamma)\ &\Leftrightarrow\ \Gterms{a:A},
\end{align*}
where $A=p\circ a$, as indicated in:
\[
\xymatrix{
 	&  \UU \ar[d]^{p}\\
\y\Gamma \ar[ru]^{a}   \ar[r]_{A}   & \U .}
\]

Thus we regard \U\ as the \emph{presheaf of types}, with $\U(\G)$ the set of all types in context $\G$, and \UU\ as the \emph{presheaf of terms}, with $\UU(\G)$ the set of all terms in context $\G$, while the component $p_{\G} : \UU(\G) \to \U(\G)$ is the typing of the terms in context $\G$ (cf.~\cite{Hofmann2} for an early statement of this point of view).

Observe that naturality of $p : \UU\to\U$ means that for any ``substitution" $\sigma:\Delta\to\G$, we have an action on types and terms:
\begin{align*}
\Gtypes{A}\ &\Rightarrow\ \types{\Delta}{A\sigma}\\
\Gterms{a:A}\ &\Rightarrow\ \terms{\Delta}{a\sigma : A\sigma}\,.
\end{align*}
While, by functoriality, given any further $\tau: \Delta'\to\Delta$, we have
\[
(A\sigma)\tau = A(\sigma\circ\tau) \qquad (a\sigma)\tau = a(\sigma\circ\tau),
\]
as well as
\[
A1 = A \qquad a1 = a
\]
for the identity substitution $1 : \G\to\G$.

Finally, the representability of $p : \UU\to\U$ is exactly the operation of \emph{context extension}: given any $\Gtypes{A}$, by Yoneda we have the corresponding map $A : \G \to \U$, and we let $p_A: \ext{\G}{A} \to \G$ be the resulting fiber of $p$ as in \eqref{diag:rep}.  We therefore have a pullback square:
\begin{equation}\label{diag:conext}
\xymatrix{
\ext{\G}{A} \ar[d]_{p_A} \ar[r]^{q_A} \pbcorner &  \UU \ar[d]^{p}\\
\G \ar[r]_{A}   & \U ,}
\end{equation}
where the map $q_A : \ext{\G}{A}\to\UU$ determines a term 
$$\terms{\ext{\G}{A}}{q_A:Ap_A}.$$
In \eqref{diag:conext} and henceforth, we omit the $\y$ for the Yoneda embedding, letting the Greek letters serve to distinguish representable presheaves.
\medskip

The fact that \eqref{diag:conext} is a pullback means that given any $\sigma: \Delta\to\G$ and $\terms{\Delta}{a:A\sigma}$, there is a map $$(\sigma, a):\Delta\to\ext{\G}{A},$$ and this operation satisfies the equations
\begin{align*}
p_A\circ(\sigma,a) &= \sigma \\
q_A(\sigma,a) &= a,
\end{align*}
as indicated in the following diagram.
\begin{equation*}
\xymatrix{
\Delta  \ar@/{}_{1pc}/[ddr]_{\sigma} \ar@{.>}[dr]|-{(\sigma,a)} \ar@/{}^{1pc}/[drr]^{a} \\
& \ext{\G}{A} \ar[d]^{p_A} \ar[r]_{q_A} \pbcorner &  \UU \ar[d]^{p}\\
& \G \ar[r]_{A}   & \U }
\end{equation*}
Moreover, by the uniqueness of $(\sigma,a)$, for any $\tau: \Delta'\to\Delta$, we also have:
\begin{align*}
(\sigma,a)\circ\tau &= (\sigma\circ\tau,a\tau)\\
(p_A,q_A) &= 1.
\end{align*}

Comparing the foregoing with the definition of a category with families  in \cite{CwF}, we have shown:\footnote{
Since completing this paper, the author has learned that the following fact was also observed independently by M.~Fiore, see \cite{fiore}.}

\begin{proposition}
Let $p : \UU\to\U$ be a natural transformation of presheaves on a small category \C\ with a terminal object.  Then $p$ is representable in the sense of Definition \ref{def:rep} just in case $(\C, p)$ is a category with families.
\end{proposition}

The notion of a category with families is a variable-free way of presenting dependent type theory, including contexts and substitutions, types and terms in context, and context extension.  Accordingly, we may think of a representable map of presheaves on a category \C\ as a ``type theory over~\C" --- with  \C\ serving as the category of contexts and substitutions (the requirement that \C\ should have a terminal object, representing the ``empty context", is purely conventional).  As we shall see below, such a map of presheaves is essentially determined by a class of maps in \C\ that is closed under all pullbacks; these maps will be the types in context.

\begin{definition}
By a \emph{natural model of type theory} on a small category \C\ we mean a representable map of presheaves, 
\[
p : \UU \to \U.
\]
\end{definition}

\begin{corollary}
Natural models of type theory are evidently closed under composition, coproducts, and pullbacks along arbitrary maps $\U'\to\U$.
\end{corollary}

\subsection{Algebraic character}\label{subsection:algebraic}
As for many concepts, the notion of representability of a natural transformation $f : A \to B$ may be regarded in either of two ways: as a property of the map: for each $C\in\C$ and $x\in X(C)$, one can find a $D\in\C$, a $p : D\to C$, and a $y\in Y(D)$ such that \dots; or as a structure on the map: there is an explicitly given function that chooses the $(D\in\C, p : D\to C,y\in Y(D))$ for each $(C\in\C, x\in X(C))$ in such a way that \dots.  Since the values of the function are uniquely determined up to specific isomorphisms, and most of the required constructions on those values respect isomorphisms, it generally makes no difference which notion is assumed, and so the more flexible ``property'' version is often more convenient.  However, there are some steps below where the structured version is required, and for that reason we shall henceforth assume that a representable natural transformation includes a specific representability structure, consisting of a ``canonical'' pullback square of the form \eqref{diag:rep} for every $\y{C}\to X$ (no coherence).  We will not dwell on the choice of structure, however, and will recall this assumption where it is required.

The situation is entirely analogous to that of  a ``category with binary products''; and indeed, as in that case, one can understand the ``property versus structure'' difference as the condition that a cetain functor should have an adjoint versus the selection of a specific adjoint (among all isomorphic options).  The latter approach has the advantage of making explicit an algebraic structure that is perhaps not evident in the former. Indeed, as has recently been emphasized (e.g.\ in \cite{BCH}), the notion of a \emph{category with families} is essentially algebraic,  consisting of the four sorts: \emph{contexts, substitutions, types, terms}; operations defined on the sorts, including in particular \emph{context extension}; and equations between terms built from the operations.  The same is, of course, true of the equivalent concept of a \emph{natural model of type theory}, in the form of a representable natural transformation $f : A \to B$ over a category $\C$, as we now briefly indicate.

Regarded as a many-sorted algebraic theory, a natural model consists of four basic sorts $$C_0,\ C_1,\ A,\ B$$ along with the following operations and equations:
\begin{description}
\item[category:] the usual domain, codomain, identity and composition arrows for the index category:
\[
\xymatrix{
C_1 \times_{C_0} C_1 \ar[r]^-{\circ} & C_1 {\ar@<-2ex>[rr]_{\dom} \ar@<2ex>[rr]^{\cod}} && \ar[ll]|-{\,\identity\,} C_0
}
\]
together with the familiar equations for a category.

\item[presheaf:] the indexing and action operations for the presheaves: 
\[
\xymatrix{
C_1 \times_{C_0} A  \ar[r]^-{\alpha} & A \ar[d]^{p_A} \\
& C_0
}
\qquad
\xymatrix{
C_1 \times_{C_0} B \ar[r]^-{\beta} & B \ar[d]^{p_B} \\
& C_0
}
\]
together with the equations making $\alpha$ a (contravariant) action of $C_1$ on $A$:
\begin{align*}
 p_A(\alpha(u, a)) &= \dom(u),\\
 \alpha(u\circ v, a) &= \alpha(v,\alpha(u,a)),\\
 \alpha(1_{p_{A}(a)},a) &= a,
\end{align*}
and similarly for $\beta$.

\item[natural transformation:] an operation 
\[
f : A \to B
\]
satisfying the naturality equations: $$p_B \circ f = p_A,\qquad f\circ\alpha = \beta\circ(C_1\times_{C_0}f).$$

\item[representable:] note that a natural transformation $f : A \to B$ is representable just if the associated functor on the categories of elements,
\[
\textstyle{\int_\C{f} : \int_\C{A} \to \int_\C{B}}
\]
has a right adjoint $f^* : \int_\C{B} \to \int_\C{A}$ (cf.~\cite{ABSS}, \S\,8), which is an algebraic condition.  

In more detail, we requiring the following additional structure:

\begin{itemize}
\item an operation
\[
(f^*)_0 : B \to A
\]
taking the objects of $\int_\C{B}$ to  those of $\int_\C{A}$ (not necessarily preserving the indexing over $C_0$), 

\item an operation on the arrows in the categories of elements:
\[
(f^*)_1 : C_1 \times_{C_0} B \to C_1 \times_{C_0} A
\]
respecting domains and codomains,
\begin{align*}
\dom(\pi_1( (f^*)_1 (u, b) )) &= (f^*)_0(\dom(u)),\\
 p_A(\pi_2( (f^*)_1 (u, b) )) &= (f^*)_0(\cod(u)),
\end{align*}
and satisfying the functoriality equations, 
\begin{align*}
 (f^*)_1((u,b)\circ(u', b')) &=  (f^*)_1(u,b)\circ (f^*)_1(u', b'),\\
 (f^*)_1(1,b)&= (1,b).
\end{align*}

\item two further operations 
\begin{align*}
\eta &: A \to C_1 \times_{C_0} A \\
\epsilon &: B \to C_1 \times_{C_0} B 
\end{align*}
satisfying the standard equations for natural transformations of the form $\eta: 1 \to f^*\circ f$ and $\epsilon: f\circ f^*\to 1$.
\item the familiar triangle identities for an adjunction.
\end{itemize}
The remaining details are left to the reader.
\end{description}

\section{Modelling the type constructors}\label{sec:modelling}

When does a natural model of type theory also model the various type constructors, such as (dependent) product $\Pi$, sum $\Sigma$, and identity types $\Id$? As the notation $p : \UU\to\U$ suggests, the notion of a natural model is similar to Voevodsky's notion of a \emph{universe} \cite{KLV}, and we shall modify the approach taken there in order to determine conditions ensuring that the usual type-forming operations are modeled in our setting. (Related ideas were used already in \cite{Fu, StreicherUp}.)

We require the following preliminary observations regarding polynomial functors, for more on which see \cite{GK}.  

Given a map $f : B \to A$  in a locally cartesian closed category $\CC$, there is an associated \emph{polynomial endofunctor} $P_f : \CC \to \CC$,
defined for every object $X\in\CC$ by
\begin{equation}\label{eq:polydef}
P_f(X)\ =\ \sum_{a:A}X^{B_a}
\end{equation}
where, as usual, we write $B_a = f^{-1}(a)$ for the fiber of $f$ at $a:A$, using the internal language of $\CC$ as explained in \cite{GK}.  Formally, this functor is defined from the LCCC structure on $\CC$ as a composite:
\[
P_f(X)\ =\ \sum_{A}\prod_{f}B^{*}(X)
\]
where:
\begin{align*}
B^* &: \CC \to \CC/B && \text{is pullback along $B\to 1$},\\
\prod_{f} &: \CC/B \to \CC/A && \text{is right adjoint to pullback along $f : B\to A$},\\
\sum_{A} &: \CC/A \to \CC && \text{is composition along $A\to 1$}.
\end{align*}

The following lemma can be found, in essentially the same form, in \cite{DT,PTJ:ppbdht}.
 
\begin{lemma}\label{lem:polyclass}
There is a natural bijection between maps $g : Y \to \sum_{a:A}X^{B_a}$ and pairs of maps $\big(g_1:Y \to A,\ g_2 : Y\times_A B \to X\big)$ as indicated in the following diagram.
\begin{equation}\label{diag:polyclass}
\xymatrix{
X & Y\times_A B \ar[l]_-{g_2}  \ar[d]\ar[r]\pbcorner &  B \ar[d]^{f}\\
&Y\ar[r]_{g_1}   & A}
\end{equation}
\end{lemma}
\begin{proof}
Given
\[
\xymatrix{
g: Y \ar[r] &  \sum_{a:A}X^{B_a} = \sum_{A}\prod_{f}B^{*}(X)
}
\]
compose with the projection $\pi_1 : \sum_{A}\prod_{f}B^{*}(X)\to A$ to get $g_1 = \pi_1\circ g : Y\to A$, making $g$ a map over $A$,
\[
\xymatrix{
Y \ar[r]^-{g}  \ar[rd]_{g_1}  &  \sum_{A}\prod_{f}B^{*}(X) \ar[d]^{\pi_1}  \\
& A.
}
\]
As an object over $A$, the map $\pi_1$ is
\[
\prod_{f}B^*(X) = \prod_{f}f^*A^*(X) = (A^*(X))^f.
\]
We can therefore take the exponential transpose of $g$ to get another map over $A$ of the form:
\[
\xymatrix{
Y\times_A B \ar[r]^{\tilde{g}}  \ar[rd]_{g_1\times f}  &  A^*(X) \ar[d]  \\
& A.
}
\]
Composing $\tilde{g}$ with the second projection $A^*(X) = A\times X \to X$ gives $g_2 :Y\times_A B \to X$ as indicated in
\begin{equation}\label{diag:compwithproject}
\xymatrix{
Y\times_A B \ar[r]^{\tilde{g}}  \ar[rd]_{g_1\times f} \ar@/{}^{20pt}/[rr]^{g_2} &  A \times X \ar[d] \ar[r] & X \\
& A .
}
\end{equation}
This assignment of $(g_1, g_2)$ to $g$ is clearly reversible and natural in $Y$.
\end{proof}

Since the isomorphism of lemma \ref{lem:polyclass} is natural in $Y$, it is convenient to consider the generic case, where $Y = \sum_{a:A}X^{B_a}$ and $g$ is the identity.  In that case, we have a diagram of the form,
\[
\xymatrix{
X & \\
G \ar[u]_{u_2} \ar[d] \ar[r] \pbcorner &  B \ar[d]^{f}\\
\sum_{a:A}X^{B_a} \ar[r]_-{u_1}   & A
}
\]
where $u_1$ is the canonical projection $\pi_1 : \sum_{a:A}X^{B_a} \to A$, and the ``generic pulled-back object" $G$ can be described over $A$ as
$X^{B_a} \times_A B_a$.
The map $u_2 : G \to X$ is then evaluation over $A$, composed with the second projection as in \eqref{diag:compwithproject}.

Now given any $g : Y \to \sum_{a:A}X^{B_a}$, the associated maps $g_1 : Y \to A$ and $g_2 : Y\times_A B \to X$ are given by pullback and composition, as indicated in the following diagram.
\begin{equation}\label{diag:polyclassnat}
\xymatrix{
 & X & \\
 Y\times_A B \ar[r] \ar[ur]^{g_2} \ar[d] \pbcorner & G \ar[u]_{u_2} \ar[d] \ar[r] \pbcorner &  B \ar[d]^{f}\\
Y \ar[r]^-{g} \ar@/{}_{20pt}/[rr]_{g_1} & \sum_{a:A}X^{B_a} \ar[r]^-{u_1}   & A
}
\end{equation}

Now consider the case of the polynomial functor $P = P_p$ of a natural model $p : \UU \to \U$, with the form
\begin{equation}\label{eq:polyuniv}
P(X)\ =\ \sum_{\term{A}{\U}}X^{A},
\end{equation}
where we write simply $A = \UU_A$ for the fiber of $p$ over $A:\U$.  This is justified by considering the pullback
\begin{equation}\label{diag:conext2}
\xymatrix{
\ext{1}{A} \ar[d] \ar[r] \pbcorner &  \UU \ar[d]^{p}\\
1 \ar[r]_{A}   & \U ,}
\end{equation}
as the case of \eqref{diag:conext} where $\G =1$ is terminal, and therefore $\ext{1}{A} = A$ is just an object of \C, i.e.\ a ``closed type".

Applying Lemma \ref{lem:polyclass} to \eqref{eq:polyuniv} in the case $f = p$ and $X=\U$ and $Y = \G$ representable, we obtain a natural, bijective correspondence:
\begin{equation}\label{class:cont}
\begin{prooftree}
\[
	(A,B): \G \to \sum_{\term{A}{\U}}\U^{A}
\Justifies
\thickness=1em
	\xymatrix{
	\U & \ext{\G}{A} \ar[l]_{B}  \ar[d]\ar[r]\pbcorner &  \UU \ar[d]^{p}\\
		&\G \ar[r]_{A}   & \U
		}
		\]
\Justifies
\thickness=1em
	\Gtypes{A},\quad \types{\ext{\G}{A}}{B}
\end{prooftree}
\end{equation}
Thus just as $\U$ \emph{classifies types in context} $\Gtypes{A}$, we can say that $P(\U)\ =\ \sum_{\term{A}{\U}}\U^{A}$ \emph{classifies types in an extended context} $\types{\ext{\G}{A}}{B}$.  For the record:

\begin{proposition}\label{prop:polyclasstypes}
The presheaf $P(\U)\ =\ \sum_{\term{A}{\U}}\U^{A}$ classifies types in context $\types{A}{B}$, in the sense that there is a natural isomorphism between maps $\G \to \sum_{\term{A}{\U}}\U^{A}$ and pairs $\types{\G}{A}$ and $\types{\ext{\G}{A}}{B}$, as displayed in the following diagram.
\begin{equation}\label{diag:polyclasstypes}
\xymatrix{
 & \U & \\
 \ext{\G}{A} \ar[r] \ar[ur]^{B} \ar[d] \pbcorner & \cdot \ar[u] \ar[d] \ar[r] \pbcorner &  \UU \ar[d]^{p}\\
\G \ar[r] \ar@/{}_{20pt}/[rr]_{A} & \sum_{\term{A}{\U}}\U^{A} \ar[r]   & \U
}
\end{equation}
\end{proposition}
\proof{
This is just diagram \eqref{diag:polyclassnat} specialized to the present case.
}

\begin{remark}
Throughout this section, we are using the same methodology as in \cite{KLV}(section 1.4), pioneered by Hofmann in \cite{Hofmann}(section 4), to reduce certain general type constructions to generic cases on a universe (cf.~also Fu \cite{Fu}).
\end{remark}
\subsection{Products}

\begin{proposition}\label{prop:prod} 
Let $P(X) =\sum_{\term{A}{\U}}X^{A}$ be the polynomial functor associated to a natural model $p : \UU \to \U$.
Then the type-theoretic rules for (extensional) dependent products are modelled by maps of the form
\begin{align}
\lambda :&\ P(\UU)\to \UU \label{prop:prod1}\\
\Pi :&\ P(\U) \to \U \label{prop:prod2}
\end{align}
making the following diagram a pullback. 
\begin{equation}\label{diag:prod}
\xymatrix{
P(\UU)  \ar[d]_{P(p)} \ar[r]^-{\lambda} &  \UU \ar[d]^{p}\\
P(\U) \ar[r]_-{\Pi} & \U }
\end{equation}
\end{proposition}

\begin{proof}
Replacing $P$ by its definition, we obtain a diagram of the form:
\begin{equation}\label{diag:prod2}
\xymatrix{
\sum_{A\ty\U}\UU^{A}  \ar[d]   \ar[rr]^-{\lambda} &&  \UU \ar[d]^{p}\\
\sum_{A\ty\U}\U^{A} 		    \ar[rr]_-{\Pi} && \U \,.}
\end{equation}
Using proposition \ref{prop:polyclasstypes}, which states that $\sum_{A\ty\U}\U^{A}$ classifies pairs $\Gtypes{A} $ and $\types{\ext{\G}{A}}{B}$, the operation $\Pi: \sum_{\term{A}{\U}}\U^{A}  \to \U$ is seen to be the type-theoretic \emph{formation rule},
\[\tag{\text{$\Pi$-form}}
\begin{prooftree}
\Gtypes{A} \qquad \types{\ext{\G}{A}}{B}
\justifies
\Gtypes{\prod_{A}{B}}.
 \end{prooftree}
\]
Now just as $P(\U)\ =\ \sum_{A:\U}\U^{A}$ classifies pairs of the form $$\Gtypes{A}, \quad \types{\ext{\G}{A}}{B},$$ so $P(\UU)\ =\ \sum_{A\ty\U}\UU^{A}$ classifies pairs of the form $$\Gtypes{A}, \quad \terms{\ext{\G}{A}}{b:B}.$$ This follows from lemma \ref{lem:polyclass} just as did proposition \ref{prop:polyclasstypes}, but replacing the presheaf of types $\U$ by the presheaf of terms $\UU$.

Thus the operation $\lambda: \sum_{A\ty\U}\UU^{A}  \to \UU$ models the type-theoretic \emph{introduction rule},
\[\tag{\text{$\Pi$-intro}}
\begin{prooftree}
\Gtypes{A} \quad \terms{\ext{\G}{A}}{b:B}
\justifies
\Gterms{\lambda_A b:\prod_{A}B}.
 \end{prooftree}
\]

Consider the \emph{elimination rule}:
\[\tag{\text{$\Pi$-elim}}
\begin{prooftree}
\Gterms{f:\prod_{A}B} \qquad \Gterms{a:A}
\justifies
\Gterms{f(a):B[a]}
 \end{prooftree}
\]
and the associated \emph{computation rule} ($\beta$):
\[\tag{\text{$\Pi$-comp $\beta$}}
\begin{prooftree}
\terms{\ext{\G}{A}}{b:B} \qquad \Gterms{a:A}
\justifies
\Gterms{(\lambda_A b)(a) = b[a] : B[a]\,.}
 \end{prooftree}
\]
The notation $B[a]$ and $b[a]$ is interpreted as follows: given $\Gterms{a:A}$, we have
\[
\xymatrix{
 	&  \UU \ar[d]^{p}\\
\Gamma \ar[ru]^{a}   \ar[r]_{A}   & \U\,,}
\]
and so by taking a pullback, we get a substitution $(1,a) : \G\to \ext{\G}{A}$ into the context extension $\ext{\G}{A}$:
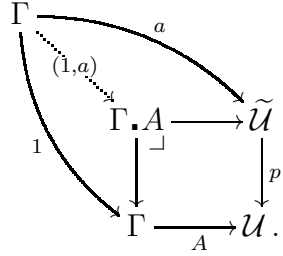
\begin{equation*}
\xymatrix{
\G  \ar@/{}_{1pc}/[ddr]_{1} \ar@{.>}[dr]|-{(1,a)} \ar@/{}^{1pc}/[drr]^{a} \\
& \ext{\G}{A} \ar[d] \ar[r] \pbcorner &  \UU \ar[d]^{p}\\
& \G \ar[r]_{A}   & \U \,.}
\end{equation*}
Now $\types{\ext{\G}{A}}{B}$ and $\terms{\ext{\G}{A}}{b:B}$ are of the form
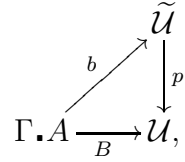
\[
\xymatrix{
 	&  \UU \ar[d]^{p}\\
\ext{\G}{A} \ar[ru]^{b}   \ar[r]_{B}   & \U,}
\]
so we can set
\begin{align*}
B[a] &= B\circ (1, a) \\
b[a] &= b\circ (1, a),
\end{align*}
as indicated in
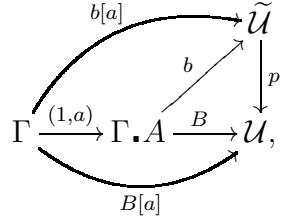
\[
\xymatrix{
 	&&  \UU \ar[d]^{p}\\
\G \ar@/{}^{20pt}/[rru]^{b[a]} \ar[r]^-{(1, a)} \ar@/{}_{20pt}/[rr]_{B[a]} & \ext{\G}{A} \ar[ru]^{b}   \ar[r]^{B}   & \U,}
\]
to get the terms
\begin{align*}
\Gtypes{B[a]} \\
\Gterms{b[a] : B[a]}.
\end{align*}

Now suppose that \eqref{diag:prod2} is a pullback.  We require a term $\Gterms{f(a):B[a]}$, assuming we have the premises $\Gterms{f:\prod_{A}B}$ and $\Gterms{a:A}$.  The first premise means there are maps $(A, B)$ and $f$ as indicated in
\begin{equation}\label{diag:pielim1}
\xymatrix{
\G  \ar@/{}_{1pc}/[ddr]_{(A,B)} \ar@{.>}[dr]|-{(A,\tilde{f})} \ar@/{}^{1pc}/[drr]^{f} \\
& \sum_{A\ty\U}\UU^{A}  \ar[d]^{P(p)} \ar[r]_-{\lambda} \pbcorner &  \UU \ar[d]^{p}\\
&\sum_{A\ty\U}\U^{A}  \ar[r]_-{\Pi}   & \U .}
\end{equation}
Since the square is a pullback, there is a map $(A,\tilde{f})$ as indicated, and by the classifying property of $\sum_{A\ty\U}\UU^{A}$, it corresponds  uniquely to a term $\terms{\ext{\G}{A}}{\tilde{f}:B}$,
\[
\xymatrix{
 	&  \UU \ar[d]^{p}\\
\ext{\G}{A} \ar[ru]^{\tilde{f}}   \ar[r]_{B}   & \U.}
\]
Now set:
\[
f(a) = \tilde{f}[a] = \tilde{f}\circ (1, a),
\]
so that indeed $\terms{\G}{f(a) : B[a]}$, as required.

For the computation rule ($\beta$), suppose $\terms{\ext{\G}{A}}{b:B}$.  Then $\Gterms{\lambda_A b: \prod_{A}B}$ and diagram \eqref{diag:pielim1} becomes
\begin{equation}\label{diag:pielim2}
\xymatrix{
\G  \ar@/{}_{1pc}/[ddr]_{(A,B)} \ar@{.>}[dr]|-{(A,\widetilde{\lambda_A b})} \ar@/{}^{1pc}/[drr]^{\lambda_A b} \\
& \sum_{A\ty\U}\UU^{A}  \ar[d]^{P(p)} \ar[r]_-{\lambda} \pbcorner &  \UU \ar[d]^{p}\\
&\sum_{A\ty\U}\U^{A}  \ar[r]_-{\Pi}   & \U \,,}
\end{equation}
for some $\terms{\ext{\G}{A}}{\widetilde{\lambda_A b}:B}$,
\[
\xymatrix{
 	&  \UU \ar[d]^{p}\\
\ext{\G}{A} \ar[ru]^{\widetilde{\lambda_A b}}   \ar[r]_{B}   & \U.}
\]
But clearly $\terms{\ext{\G}{A}}{b:B}$ also satisfies the condition $\lambda\circ(A,b)=\lambda_Ab$, so by the universal property of the pullback, we have
\[
\terms{\ext{\G}{A}}{\widetilde{\lambda_A b} = b:B}\,.
\]
But then
\[
\Gterms{(\lambda_A b)(a) = (\widetilde{\lambda_A b})[a] = b[a] : B[a]}\,,
\]
as required.

An additional computation rule is required for \emph{extensional} $\Pi$-types, and it is also satified, namely the so-called $\eta$-rule.  This rule is written with variables in the form ${{\lambda}x:A}.f(x) = f\ty \prod_{A}B$.  In the variable-free style of categories with families, this takes the form (cf.~\cite{CwF}, 2.2):

\[\tag{\text{$\Pi$-comp $\eta$}}
\begin{prooftree}
\Gterms{f:\prod_{A}B}
\justifies
\Gterms{\lambda_A (fp_A(q_A)) = f : \prod_{A}B}\,,
 \end{prooftree}
\]
where, recall from \eqref{diag:conext}, that for any $\Gtypes{A}$ the terms $p_A$ and $q_A$ are defined by
\begin{equation*}
\xymatrix{
\ext{\G}{A} \ar[d]_{p_A} \ar[r]^{q_A} \pbcorner &  \UU \ar[d]^{p}\\
\G \ar[r]_{A}   & \U\,.
}
\end{equation*}
We have:
\[
(fp_A)(q_A) = \widetilde{(fp_A)}[q_A] = \widetilde{(fp_A)}(1, q_A) = \widetilde{f}p_A(1, q_A) = \widetilde{f}.
\]
Therefore 
\[
\lambda_A (fp_A(q_A)) = \lambda\circ(A, fp_A(q_A)) = \lambda\circ(A, \widetilde{f}) = f.
\]
The straightforward verification of the converse implication is omitted.
\end{proof}

Examining the proof (along with some further reasoning), we see that a more precise formulation of Proposition \ref{prop:prod} is possible: 

\begin{corollary}
The type-theoretic formation and introduction rules for dependent products are modelled by maps of the form
\begin{align}
\lambda :&\ P(\UU)\to \UU \label{prop:prod1}\\
\Pi :&\ P(\U) \to \U \label{prop:prod2}
\end{align}
making the following square commute. 
\begin{equation}\label{diag:prod}
\xymatrix{
P(\UU)  \ar[d]_{P(p)} \ar[r]^-{\lambda} &  \UU \ar[d]^{p}\\
P(\U) \ar[r]_-{\Pi} & \U }
\end{equation}
The square is a weak pullback, with a distinguished section of the canonical map $P(\UU) \to P(\U) \times_{\U}\UU$, if and only if the elimination and $\beta$ computation rules hold, and a pullback if and only if in addition the $\eta$ computation rule holds.
\end{corollary}
A similar strengthening is also possible for the following treatment of the type constructors $\Sigma$ and $\Id$.

\begin{remark}\label{remark:substitution}
The type and term constructors $\Pi$, $\lambda$, and $-_1(-_2)$ occurring in the formation, introduction, and elimination rules are also required to respect substitutions $\sigma : \Delta\to \G$.  Specifically, consider e.g.\ the $\Pi$-formation rule:
\begin{equation}\label{eq:piform}
\begin{prooftree}
\Gtypes{A} \qquad \types{\ext{\G}{A}}{B}
\justifies
\Gtypes{\prod_{A}{B}}.
 \end{prooftree}
\end{equation}
Applying $\sigma : \Delta\to \G$ to the premises gives a new instance of the rule
\[
\begin{prooftree}
\types{\Delta}{A\sigma} \qquad \types{\ext{\Delta}{A\sigma}}{B\sigma}
\justifies
\types{\Delta}{\prod_{A\sigma}{B\sigma}}.
 \end{prooftree}
\]
On the other hand, one can instead apply $\sigma : \Delta\to \G$  to the conclusion of \eqref{eq:piform}
 to obtain
 \[
 \types{\Delta}{(\prod_{A}{B})\sigma}.
 \]
It is part of the definition of ``modelling the product rules in a category with families" that these two things should be the same,
\[
\types{\Delta}{(\prod_{A}{B})\sigma = \prod_{A\sigma}{B\sigma} } 
\]
as elements of $\U(\Delta)$.  But indeed, we have
\begin{align*}
(\prod_{A}{B})\sigma &= (\Pi \circ (A,B))\circ\sigma \\
	&= \Pi \circ ((A,B)\circ\sigma) \\
	&= \Pi \circ (A\sigma, B\sigma) \\
	&= \prod_{A\sigma}{B\sigma} 
\end{align*}
as indicated in the following diagram
\[
\xymatrix{
 \Delta \ar@/{}_{15pt}/[rd]_-{(A\sigma,B\sigma)} \ar[r]^{\sigma} & \G \ar[d]_-{(A,B)} \ar@/{}^{10pt}/[rd]^-{\prod_{A\sigma}{B\sigma}}& \\
& \sum_{A\ty\U}\U^{A} \ar[r]_-{\Pi}   & \U\,,}
\]
where the equation 
\[
(A,B)\circ\sigma \ = \ (A\sigma, B\sigma) 
\]
follows easily from proposition \ref{prop:polyclasstypes}.  The other two required equations,
\begin{align*}
(\lambda_{A}{b})\sigma &= \lambda_{A\sigma}({b\sigma}) \\
(f(a))\sigma &= (f\sigma)(a\sigma)
\end{align*}
follow similarly.
\end{remark}

\begin{remark}
Our specification differs from that in \cite{KLV} only in the specification of the objects on the left side of the square \eqref{diag:prod} via the polynomial functor $P(X) =\sum_{\term{A}{\U}}X^{A}$, rather than by applying the $\Pi$ constructor in presheaves to the generic case, as in the ``logical framework'' approach. 
\end{remark}

\subsection{Sums}

For the sum constructor $\Sigma$ we shall replace the family 
\[
\xymatrix{
\sum_{A\ty\U}\UU^{A} \ar[r] & \sum_{A\ty\U}\U^{A}
}
\]
 in diagram \eqref{diag:prod2}
 by a different one, corresponding to the different premises of the $\Sigma$-\emph{introduction rule},
\[\tag{\text{$\Sigma$-intro}}
\begin{prooftree}
\Gtypes{A} \qquad \types{\ext{\G}{A}}{B} \qquad \Gterms{a:A} \qquad \Gterms{b:B[a]}
\justifies
\Gterms{\pair{a, b}:\sum_{A}B}.
 \end{prooftree}
\]
The base object $\sum_{A\ty\U}\U^{A}$ remains the same, corresponding to the fact that the $\Sigma$-\emph{formation rule} has the same form as the one for $\Pi$, namely,
\[\tag{\text{$\Sigma$-form}}
\begin{prooftree}
\Gtypes{A} \qquad \types{\ext{\G}{A}}{B}
\justifies
\Gtypes{\sum_{A}{B}}.
 \end{prooftree}
\]
But the object over $\sum_{A\ty\U}\U^{A}$ must now classify data of the form
\begin{equation}\label{eq:sigmadata}
\big( \Gtypes{A},\ \types{\ext{\G}{A}}{B},\ \Gterms{a:A},\ \Gterms{b:B[a]} \big).
\end{equation}
This is accomplished with the following object (again constructed using the internal language in presheaves):
\[
\sum_{(A\ty\U)}\sum_{(B\ty\U^{A})}\sum_{(a : A)}B(a) \,.
\]
We have a projection $\pi$ associated to the first two sums,
\[
\xymatrix{
\sum_{(A\ty\U)}\sum_{(B\ty\U^{A})}\sum_{(a : A)}B(a) \ar[r]^{\cong} \ar[rd]_{\pi} & \sum_{(A, B)\ty(\sum_{A\ty\U}\U^{A})} \sum_{(a : A)}B(a) \ar[d] \\
& \sum_{A\ty\U}\U^{A}\,,
}
\]
and factorizations of maps of the form $(A,B):\G \to \sum_{A\ty\U}\U^{A}$ through $\pi$,
\[
\xymatrix{
& \sum_{(A\ty\U)}\sum_{(B\ty\U^{A})}\sum_{(a : A)}B(a)  \ar[d]^{\pi}\\
\G \ar[r]_{(A,B)} \ar@{.>}[ru] & \sum_{A\ty\U}\U^{A} }
\]
are then in natural, bijective correspondence with data of the form \eqref{eq:sigmadata}, as can  be proved similarly to proposition \ref{prop:polyclasstypes}.

\begin{proposition}\label{prop:sum} 
Given a natural model $p : \UU \to \U$, the type theoretic rules for (extensional) dependent sum are modelled by maps  
\begin{align}\label{eq:sumop1}
\pairmap :& \sum_{(A\ty\U)}\sum_{(B\ty\U^{A})}\sum_{(a : A)}B(a) \to \UU\\
\Sigma :& \sum_{A\ty\U}\U^{A} \to \U\label{eq:sumop2}
\end{align}
making the following diagram a pullback. 
\begin{equation}\label{diag:sigma}
\xymatrix{
 \sum_{(A\ty\U)}\sum_{(B\ty\U^{A})}\sum_{(a : A)}B(a) \ar[d]_{\pi} \ar[rr]^-{\pairmap} &&  \UU \ar[d]^{p}\\
\sum_{A\ty\U}\U^{A} \ar[rr]_-{\Sigma} && \U }
\end{equation}
\end{proposition}
%

\begin{proof}
The operations \eqref{eq:sumop1} and \eqref{eq:sumop2} clearly give the $\Sigma$-introduction and formation rules, respectively.
We shall prove the \emph{extensional} elimination rule, which has the two parts
\[\tag{\text{$\Sigma$-elim}}
\begin{prooftree}
\Gterms{c:\sum_{A}B}
\justifies
\Gterms{\pi_1(c) : A}
 \end{prooftree}
\qquad\qquad
\begin{prooftree}
\Gterms{c:\sum_{A}B}
\justifies
\Gterms{\pi_2(c) : B[\pi_1(c)]}
 \end{prooftree}
\]
with associated $\Sigma$-\emph{computation rules}:
\begin{align*}\tag{\text{$\Sigma$-comp}}
\pi_1(\pair{a,b}) &= a \\
\pi_2(\pair{a,b}) &= b \\
\pair{\pi_1(c), \pi_2(c)} &= c\,.
\end{align*}
%
To show this, assume that \eqref{diag:sigma} is a pullback, let $\G$ be any object, and suppose that we have $(A, B): \G\to (\sum_{A\ty\U}\U^{A})$ and $c : \G\to\UU$ such that $\Sigma\circ a = p\circ c : \G\to\U$, which means exactly that $\terms{\G}{c : \sum_{A}B}$.  There is then a unique map
\[
\tilde{c} : \G \to  \sum_{(A\ty\U)}\sum_{(B\ty\U^{A})}\sum_{(a : A)}B(a)
\]
with $\pi\circ\tilde{c} = (A, B)$ and $\pairmap\circ\tilde{c} = c$.  Since $\tilde{c}$ is known to be uniquely of the form \eqref{eq:sigmadata}, we can write it as $$\tilde{c} = \big(A, B, \pi_1(c), \pi_2(c)\big)$$ with $\pi_1(c):A$ and $\pi_2(c): B[\pi_1(c)]$.  We then write $$\pairmap\big(A, B, \pi_1(c), \pi_2(c)\big) = \pair{\pi_1(c), \pi_2(c)}$$ accordingly.  Thus indeed $c = \pair{\pi_1(c), \pi_2(c)}$, as required.  To prove the other two $\Sigma$-computation equations, it suffices by the uniqueness of elements classified to show that, for any $a:A$ and $b:B[a]$, we have 
$$(A, B, a, b) = \big(A, B, \pi_1(\pair{a,b}), \pi_2(\pair{a,b})\big).$$
  But this is now clear, since 
\begin{align*}
\pairmap(A, B, a, b) &= \pair{a,b}\\
	&=  \pair{\pi_1(\pair{a,b}), \pi_2(\pair{a,b})}\\
	&=  \pairmap(A, B, \pi_1(\pair{a,b}), \pi_2(\pair{a,b})).
\end{align*}
Again, the converse is just as direct. 
\end{proof}

As in the case of products, the operations $\Sigma$, $\pi_1, \pi_2$, and $\pairmap$ can easily be shown to respect substitution $\sigma: \Delta\to\G$.
 
\begin{remark}
The map 
\begin{equation}
\xymatrix{
 \sum_{(A\ty\U)}\sum_{(B\ty\U^{A})}\sum_{(a : A)}B(a) \ar[d]_{\pi} \\
\sum_{A\ty\U}\U^{A}
}
\end{equation}
from \eqref{diag:sigma} can also be understood  in terms of polynomial functors.  As in \eqref{eq:polydef}, let $$P_p(X) = \sum_{A\ty\U}X^{A}$$ be the polynomial functor $P_p : \widehat{\C}\to\widehat{\C}$ determined by the map $p : \UU\to\U$.  The map $\pi$ above also determines a polynomial functor $P_\pi : \widehat{\C}\to\widehat{\C}$, again via \eqref{eq:polydef}.  These are related by
\[
P_\pi = P_p \circ P_p\ .
\]
Thus in particular the composite ${P_p}^2 = P_p\circ P_p : \widehat{\C}\to\widehat{\C}$ is also polynomial, and $\pi$ is the map representing it.  Moreover, recall from \cite{GK} that pullback diagrams of maps
\begin{equation*}\label{diag:pbpoly}
\xymatrix{
B \ar[d]_{f} \ar[r] \pbcorner &  D \ar[d]^{g}\\
A \ar[r] & C }
\end{equation*}
in $\widehat{\C}$ correspond to morphisms of the polynomial functors on $\widehat{\C}$ that they determine, $P_f \Rightarrow P_g$ (cartesian natural transformations).
Thus the pullback condition \eqref{diag:sigma} says that there is a map of polynomial functors $P_p \circ P_p \Rightarrow P_p$\,.  It is easy to see that there is also a map $1_\C \Rightarrow P_p$, determined by the terminal object of \C, with its unique term,
\begin{equation}
\xymatrix{
1 \ar[d] \ar[r] \pbcorner &  \UU \ar[d]^{p}\\
1 \ar[r] & \U .}
\end{equation}
The further investigation of this structure is left to future work (cf.~\cite{PTJ:BD}, section 2, for a related development).
\end{remark}

\begin{remark}\label{remark:catoftypes}
We mention only in passing the full internal subcategory $\mathbb{U}$ in $\widehat{\C}$ determined by $p : \UU\to\U$, which may be called the \emph{category of types}.  This presheaf of categories has $\mathbb{U}_0 = \U$ as its object of objects, and as object of arrows $\mathbb{U}_1$ the exponential
\[
\xymatrix{
B^A \ar[d] \\
\U \times \U}
\]
in $\widehat{\C}/(\U \times \U)$, where we have written $A = p_1^*(\UU)$ and $B = p_2^*(\UU)$ for the results of pulling $p : \UU \to \U$ back along the two projections $p_1, p_2 : \U \times \U \to \U$. This internal category can be seen to be cartesian closed, in virtue of the rules just given for $1$, $\Sigma$, and $\Pi$.  Indeed, the category of all types $\Gtypes{A}$ in a given context $\G$ is always cartesian closed, and substitution $\Delta \to \G$ preserves the cartesian closed structure.
\end{remark}

\subsection{Extensional identity}

The formation and introduction rues for identity types are as follows.
\[\tag{$\Id$-form}
\begin{prooftree}
\Gtypes{A} \quad 
\Gterms{a :  A}  \quad
\Gterms{b :  A} 
\justifies
\Gtypes{\id{A}(a,b)}
 \end{prooftree}
\]
\smallskip

\[\tag{$\Id$-intro}
\begin{prooftree}
\Gterms{a :  A} 
\justifies
 \Gterms{\iy(a) :  \id{A}(a,a)}
 \end{prooftree} 
\]

To interpret these, we use the ``diagonal" map $\delta:\UU\to\UU\times_\U\UU$ of a natural model $p : \UU \to \U$, formed by first taking the pullback of $p$ against itself, and then factoring the identity morphism $1:\UU\to\UU$ as indicated in the following diagram.
\[
\xymatrix{
\UU \ar@/{}_{1pc}/[ddr]_{1} \ar@{.>}[dr]|-{\delta} \ar@/{}^{1pc}/[drr]^{1} &&\\
& \UU\times_\U \UU \ar[r] \ar[d] \pbcorner & \UU \ar[d]^p \\
& \UU \ar[r]_p & \U
}
\]

\begin{proposition}\label{prop:extid} 
For a natural model $p : \UU \to \U$, the type theoretic rules for \emph{extensional} identity types are modelled by maps  
\begin{align}
\iy &: \UU \to \UU \label{prop:exteq1}\\
\Id &: \UU\times_\U \UU \to \U \label{prop:exteq2}
\end{align}
making the following diagram a pullback. 
\begin{equation}\label{diag:extid}
\xymatrix{
\UU \ar[r]^\iy  \ar[d]_-\delta & \UU \ar[d]^p \\
\UU\times_\U \UU \ar[r]_-\Id & \U
}
\end{equation}
\end{proposition}

\begin{proof}
Since maps $(A, a,b) : \G \to \UU\times_\U \UU$ correspond naturally to pairs of terms $\Gterms{a,b:A}$ of the same type $\Gtypes{A}$,
 we can set 
\[
 \id{A}(a,b) = \Id\circ (A,a,b).
\]
Then $\Id : \UU\times_\U \UU \to \U$ validates the $\Id$-formation rule.  

Moreover, given any element $a : \G \to \UU$, the commutativity of \eqref{diag:extid} means exactly that
$\Gterms{ \iy\circ a :  \id{A}(a,a)}$.  So setting
\[
\iy(a) = \iy \circ a
\]
also gives the introduction rule. The interpretation of the formation and introduction rules is then displayed by the following diagram.
\begin{equation}\label{diag:intformintroextid}
\xymatrix{
&& \UU \ar[r]^\iy  \ar[d]^-\delta & \UU \ar[d]^p \\
\G \ar@/{}^{1pc}/[rru]^a \ar[rr]^{(A,a,b)} \ar[rrd]_A && \UU\times_\U \UU \ar[d] \ar[r]_-\Id & \U \\
&& \U
}
\end{equation}

Suppose \eqref{diag:extid} is a pullback.  Then for any $(A,a,b) : \G \to \UU\times_\U \UU$ and $c: \G \to \UU$ such that $\Id(A,a,b) = p\circ c$, meaning that
\[
\Gterms{c : \id{A}(a,b)},
\]
there is a unique 
$u : \G \to \UU$
with $\delta\circ u = (A,a,b)$ and $\iy\circ u = c$.  But this means that
\[
\Gterms{a = b : A} \qquad\text{and}\qquad \Gterms{c = \iy(a) : \id{A}(a,a)}.
\]
Thus we have the standard rules for \emph{extensional} Identity types:
\[\tag{ext-$\Id$-elim}
\begin{prooftree}
\Gterms{c : \id{A}(a,b)}
\justifies
\Gterms{a = b : A}
\end{prooftree}
\qquad\qquad
\begin{prooftree}
\Gterms{c : \id{A}(a,b)}
\justifies
\Gterms{c = \iy(a) : \id{A}(a,a)}
\end{prooftree}
\]
The converse is, again, equally direct.
\end{proof}

Summing up:

\begin{theorem}\label{thm:extid}
A natural model of \emph{extensional} Martin-L\"of type theory with product, sum, and identity types is given by a small category \C\ equipped with a representable map of presheaves, $$p : \UU\to\U,$$ together with maps,
\[
\Pi,\  \lambda,\  \Sigma,\  \pairmap,\  \Id,\  \iy,
\]
as in propositions \ref{prop:prod}, \ref{prop:sum}, and \ref{prop:extid}, such that the squares \eqref{diag:prod}, \eqref{diag:sigma}, and \eqref{diag:extid} are pullbacks.
\end{theorem}

\subsection{Intensional identity}

Models of extensional type theory can be obtained easily from locally cartesian closed categories by various methods, including \cite{Hofmann}.  We are mainly interested here in models of the \emph{intensional} theory.  The formation and introduction rules remain the same as in the extensional case, but the elimination rule takes a somewhat more complex form inspired by inductive definitions.  We shall confine attention here to the modifications required for modelling intensional \emph{identity} types, but an analogous treatment is also possible for sum and product types, which we leave for future work. 

As in \eqref{prop:exteq1} and\eqref{prop:exteq2}, we assume maps
\begin{align*}
\iy &: \UU \to \UU \\
\Id &: \UU\times_\U \UU \to \U ,
\end{align*}
but now we merely require the resulting square to commute:
\begin{equation}\label{diag:intid}
\xymatrix{
\UU \ar[r]^\iy  \ar[d]_-\delta & \UU \ar[d]^p \\
\UU\times_\U \UU \ar[r]_-{\Id} & \,\U\,.
}
\end{equation}

Once again we can set 
\[
 \id{A}(a,b) = \Id\circ (A,a,b)
\]
to validate the $\Id$-formation rule:
\[\tag{$\Id$-form}
\begin{prooftree}
\Gtypes{A} \quad 
\Gterms{a :  A}  \quad
\Gterms{b :  A} 
\justifies
\Gtypes{\id{A}(a,b)}\,.
 \end{prooftree}
\]
Also as before, given any element $a : \G \to \UU$, we  have
$\Gterms{\iy\circ a :  \id{A}(a,a)}$.  So setting
\[
\iy(a) = \iy\circ a
\]
again gives the introduction rule:
\[\tag{$\Id$-intro}
\begin{prooftree}
\Gterms{a :  A} 
\justifies
 \Gterms{\iy(a) :  \id{A}(a,a)}\,.
 \end{prooftree} 
\]

%
%

Now take the pullback of $p$ along $\Id$.  We obtain an object $I$ over $\UU\times_\U \UU$, together with a factorization $\rho = (\delta, \iy)$ of the diagonal $\delta$:
\begin{equation*}
\xymatrix{
\UU \ar@/{}_{1pc}/[ddr]_{\delta} \ar@{.>}[dr]|-{\rho} \ar@/{}^{1pc}/[drr]^{\iy} &&\\
& I \ar[r]  \ar[d]  \pbcorner & \UU \ar[d]^p \\
& \UU\times_\U \UU \ar[r]_-{\Id} & \U\,.
}
\end{equation*}
This structure serves as a ``generic identity type".
%
%
Indeed, consider the following diagram, in which the parallel vertical arrows are the evident projections, and the indicated squares are constructed as pullbacks.
\begin{equation}\label{diag:big}
\xymatrix@=3em{
\UU \ar[d]_{p} & \G\exdot A \ar[l]_c \ar[d]^{\rho_A} \ar[r] \pbcorner 					& \UU \ar[d]_{\rho} \ar[dr]^{\iy} 					& \\
\U & \G\exdot A\exdot A\exdot\Id_A \ar[l]^-C \ar@{.>}[lu]_{\jay(c)} \ar[d] \ar[r] \pbcorner 	& I \ar[d] \ar[r] \pbcorner 					& \UU \ar[d]^{p} \\
& {\G\exdot A\exdot A} \ar@<-1ex>[d] \ar@<.5ex>[d] \ar[r]_{q^2_A}  \pbcorner 	& \UU\times_\U \UU \ar[r]_-{\Id} \ar@<-1ex>[d] \ar@<.5ex>[d] 	& \U \\
& \ext{\G}{A} \ar[d]_{p_A} \ar[r]_{q_A} \pbcorner 			& \UU \ar[d]^{p} 						& \\
& \G \ar[r]_A 									& \U 									&
}
\end{equation}
The interpretation of $\types{\G \exdot A \exdot A}{\Id_A}$ is the center horizontal composite 
$$\Id\circ q^2_A\ :\ \ttexdot{\G}{A}{A }\to \U,$$
and so the context extension $\qtexdot{\G}{A}{A}{\Id_A}$ is the indicated pullback of $I$ along $q^2_A$.
Observe that $$\ttexdot{\G}{A}{A} = (\btexdot{\G}{A})\times_\G (\btexdot{\G}{A}),$$ and that the map
\begin{equation}\label{eq:substrho}
\rho_A : \btexdot{\G}{A}\to \qtexdot{\G}{A}{A}{\Id_A}
\end{equation}
factors the diagonal $\delta_A : \btexdot{\G}{A}\to \ttexdot{\G}{A}{A}$, because it is the pullback of ${\rho : \UU\to I}$, which factors the diagonal $\UU\to \UU\times_\U\UU$.  The map $\rho_A$ interprets the substitution $a \mapsto (a,a,\iy(a))$ associated to the introduction term $\iy$.

We can now state the \emph{$\Id$-elimination rule} as follows:
\smallskip
\[
\tag{$\Id$-elim}
\begin{prooftree}
\Gtypes{A} \qquad \types{\G\exdot A\exdot A\exdot \id{A}}{C}
\qquad \terms{\ext{\G}{A}}{c :  C\rho_A}
\justifies
\terms{\G\!\centerdot\! A\!\centerdot\! A\!\centerdot\! \id{A}}{\jay(c):C}
\end{prooftree}
\]
where the indicated substitution $C\rho_A$ is taken along the map $\rho_A$ just defined \eqref{eq:substrho}.  

The associated \emph{computation rule} then has the form:
\[\tag{$\Id$-comp}
\begin{prooftree}
\Gtypes{A} \qquad \types{\G\exdot A\exdot A\exdot \id{A}}{C}
\qquad \terms{\ext{\G}{A}}{c :  C\rho_A}
\justifies
\terms{\G\!\centerdot\! A}{\jay(c)\rho_A = c : C\rho_A}\,.
\end{prooftree}
\]

The elimination and computation rules are interpreted in the upper left square of the diagram \eqref{diag:big}, where the dotted arrow $\jay(c)$ indicates a choice of diagonal filler interpreting the corresponding term.  Since the rules are supposed to hold for all types $C$ and terms $c$, they are evidently equivalent to the following condition.
\begin{quote}
\emph{For any $\G:\type$ and $\Gtypes{A}$, the substitution $\rho_A$ has the left-lifting property with respect to $p : \UU\to\U$\,.}
\end{quote}
Here, recall that a map $f:A\to B$ is said to have the \emph{left lifting property} with respect to another $g: C\to D$, written $$f \pitchfork g\,,$$ if every commutative square from $f$ to $g$ has at least one diagonal filler,
\[
\xymatrix{
A \ar[d]_{f} \ar[r] & C \ar[d]^{g} \\
B \ar[r] \ar@{.>}[ur] & D.
}
\]


\begin{remark}
Let us consider the requirement that the rules must respect substitution, in the sense of remark \ref{remark:substitution},  for the present case.  The formation and introduction rules clearly satisfy this condition, since they are modeled by composition with particular maps.  Indeed, consider the diagram \eqref{diag:intformintroextid}, which gives the interpretation of formation and introduction, and take any substitution $\sigma:\Delta\to\G$.
\begin{equation}\label{diag:intformintroextidsubst}
\xymatrix@=3em{
&&& \UU \ar[r]^\iy  \ar[d]^-\delta & \UU \ar[d]^p \\
\Delta \ar@/{}^{1pc}/[rrru]^{a\sigma} \ar[r]_{\sigma} \ar@/{}_{1pc}/[rrrd]_-{(A\sigma,a\sigma,b\sigma)} & \G \ar[rru]^a \ar[rr]^{(A,a,b)} \ar[rrd]^A && \UU\times_\U \UU \ar[d] \ar[r]_-\Id & \U \\
&&& \U
}
\end{equation}
As indicated in the diagram above, we then have the required conditions:
\begin{align*}
\id{A}(a,b)\sigma &= (\Id\circ(A,a,b))\sigma = \Id\circ((A,a,b)\sigma) \\
&=  \Id\circ(A\sigma,a\sigma,b\sigma) = \id{A\sigma}(a\sigma,b\sigma)\\
\iy(a)\sigma &= (\iy\circ a)\sigma = \iy\circ (a\circ \sigma) = \iy\circ (a\sigma) = \iy(a\sigma).
\end{align*}

The corresponding condition for the elimination rule has the form:
\begin{equation*}
\jay(c)\sigma = \jay(c\sigma).
\end{equation*}

More precisely, for any substitution $\sigma :\Delta\to \G$, we require that 
\begin{equation}\label{eq:BCforJ}
\jay(c)\sigma_{\Id_{A}} = \jay(c\sigma_{A}),
\end{equation}
as indicated in the following diagram.
\begin{equation*}
\xymatrix@=3em{
\Delta\exdot A\sigma \ar[d]_{\rho_A\sigma} \ar[rr]^{\sigma_A} \pbcorner 	&& \G\exdot A \ar[d]^{\rho_A} \ar[r]^c		& \UU \ar[d]^{p} \\
\Delta\exdot A\sigma\exdot A\sigma\exdot\Id_{A\sigma} \ar[rr]_-{\sigma_{\Id_{A}}} \ar@{.>}[urrr]^{\jay(c\sigma_A)} \ar[d] \ar[rr] \pbcorner && \G\exdot A\exdot A\exdot\Id_A \ar@{.>}[ur]|-{\jay(c)} \ar[r]_-C \ar[d] \pbcorner 	& \U \\
\Delta \ar[rr]_\sigma									&& \G				&
}
\end{equation*}
In the cases of $\Pi$ and $\Sigma$, the analogous condition followed from the uniqueness of a certain map into a pullback.  But in this case, there is no such uniqueness, and one must instead require the existence of a family of maps $\jay(c)$, in all situations of the form
\begin{equation}\label{diag:Idfill}
\xymatrix{
\G\exdot A \ar[d]_{\rho_A} \ar[r]^c		& \UU \ar[d]^{p} \\
\G\exdot A\exdot A\exdot\Id_A \ar@{.>}[ur]|-{\jay(c)} \ar[r]_-C & \,\U,
}
\end{equation}
and selected in such a way as to be compatible with all maps of the form $\sigma :\Delta\to \G$, in the sense of \eqref{eq:BCforJ}.

We shall take a different approach in what follows: as in the cases of $\Pi$ and $\Sigma$, we shall specify a single map $\jay$ in a suitable universal case, which then gives rise to the individual maps $\jay(c)$ in a uniform way, which is then automatically natural in the sense of \eqref{eq:BCforJ}.
\end{remark}

We require a preliminary definition.
Let $f : A\to B$ and $g : C\to D$ be maps in a  cartesian closed category $\mathcal{C}$, and consider the following square, which always commutes.
\begin{equation}\label{diag:intLLP}
\xymatrix@=3em{
C^B \ar[d]_{g^B} \ar[r]^{C^f} & C^A \ar[d]^{g^A} \\
D^B \ar[r]_{D^f} & D^A.
}
\end{equation}
Taking the pullback of $D^f$ and $g^A$, we obtain a canonical comparison map $c = (g^B, C^f)$ as indicated in the following.
\begin{equation}\label{diag:intLLP2}
\xymatrix{
C^B \ar[dd]_{g^B} \ar[rr]^{C^f} \ar@{.>}[dr]|-{\,c\,} && C^A \ar[dd]^{g^A} \\
& D^B\times_{D^A} C^A \ar[ld] \ar[ru] \pbcorner & \\
D^B \ar[rr]_{D^f} && D^A.
}
\end{equation}

\begin{definition}\label{def:LLP}
A \emph{left-lifting structure $s$ for $f$ with respect to $g$}, written
$$f\ \pitchfork_s\ g$$
is a section of the comparison map $c = (g^B, C^f)$ in \eqref{diag:intLLP2}, 
\begin{equation}
\xymatrix{
C^B \ar[rr]_-{c} && D^B\times_{D^A} C^A  \ar@/_{2pc}/@{.>}[ll]^{s}\,.
}
\end{equation}
\end{definition}

\begin{lemma}\label{lem:extintLLP}
Let $f : A\to B$ and $g : C\to D$ be maps in a locally cartesian closed category $\mathcal{C}$.
The following conditions are equivalent.
\begin{enumerate}
\item\label{lem:intLLP} $f$ has a left-lifting structure $s$ with respect to $g$, 
$$f\ \pitchfork_s\ g.$$
 
\item\label{lem:expLLP} For each object $X$ and maps $\alpha, \beta$ as indicated in the diagram below (making the outer, stretched square commute), 
\begin{equation}\label{diag:expLLP}
\xymatrix@=3em{
X \ar@/{}_{1pc}/[rdd]_{\alpha}  \ar@{.>}[rd]|-{\gamma(\alpha,\beta)} \ar@/{}^{1pc}/[rrd]^{\beta} & & \\
	& C^B \ar[d]^{g^B} \ar[r]_{C^f} 	& C^A \ar[d]^{g^A} \\
	& D^B \ar[r]_{D^f} 				& D^A.
}
\end{equation}
there is an associated map $\gamma(\alpha,\beta)$ as shown (making the evident triangles commute), and the assignment is natural in $X$ in the sense that for any $u : Y\to X$,
\[
\gamma(\alpha,\beta)\circ u = \gamma(\alpha\circ u,\beta\circ u).
\]

\item\label{lem:extLLP} For all objects $X$, $$(X\times f) \pitchfork g$$ \emph{naturally in $X$}, in the sense that there exists a family $c(a,b)$ of diagonal fillers
\[
\xymatrix{
X\times A \ar[d]_{X\times f} \ar[rr]^{a} && C \ar[d]^{g} \\
X\times B \ar[rr]_{b} \ar@{.>}[urr]|-{c(a,b)} && D
}
\]
that are natural in $X$, meaning that for every $u : Y \to X$, 
\[
c(a,b)\circ(u\times B) = c(a\circ(u\times A),b\circ(u\times B))
\]
as in the following diagram, where we have written $c = c(a,b)$ and $c' = c(a\circ(u\times A),b\circ(u\times B))$.
\begin{equation*}
\xymatrix{
Y\times A	\ar[rr]^{u\times A} \ar[d]_{Y\times f} && X\times A  \ar[d]^{^{X\times f} } \ar[r]^-a		& C \ar[d]^{g} \\
Y\times B \ar[rr]_{u\times B} \ar@{.>}[urrr]^{c'}  && X\times B  \ar@{.>}[ur]_{c} \ar[r]_-b	& D
}
\end{equation*}
\end{enumerate}
\end{lemma}
\begin{proof} (Sketch)
To show that \ref{lem:intLLP} implies \ref{lem:expLLP}, suppose $f\ \pitchfork_s\ g$ and take any $\alpha$ and $\beta$ making the following commute.
\begin{equation*}
\xymatrix@=3em{
X \ar@/{}_{1pc}/[rdd]_{\alpha}  \ar@/{}^{1pc}/[rrd]^{\beta} & & \\
	& C^B \ar[d]^{g^B} \ar[r]_{C^f} 	& C^A \ar[d]^{g^A} \\
	& D^B \ar[r]_{D^f} 				& D^A.
}
\end{equation*}
Then as $\gamma(\alpha,\beta) : X\to C^B$ we can take the composite map
\[\xymatrix{
X \ar[r]^-{(\alpha,\beta)} & D^B\times_{D^A}C^A \ar[r]^-{s} & C^B.
}
\]
Conversely, in diagram \eqref{diag:expLLP} let $X = D^B\times_{D^A}C^A$ and let $\alpha$ and $\beta$ be the projections from the pullback.  The resulting map $\gamma : D^B\times_{D^A}C^A \to C^B$ is then the required section $s$. 

Now assume condition \ref{lem:expLLP}.  To prove \ref{lem:extLLP}, take any $X$ and suppose we have $a : X\times A\to C$ and $b: X\times B\to D$ making the following commute.
\[
\xymatrix{
X\times A \ar[d]_{X\times f} \ar[rr]^{a} && C \ar[d]^{g} \\
X\times B \ar[rr]_{b} && D
}
\]
transposing, we obtain a commutative diagram
\begin{equation*}
\xymatrix@=3em{
X \ar@/{}_{1pc}/[rdd]_{\tilde{a}}  \ar@/{}^{1pc}/[rrd]^{\tilde{b}} & & \\
	& C^B \ar[d]^{g^B} \ar[r]_{C^f} 	& C^A \ar[d]^{g^A} \\
	& D^B \ar[r]_{D^f} 				& D^A.
}
\end{equation*}
Thus there is a map $\gamma(\tilde{a},\tilde{b}) :  X \to C^B$ making the evident triangles commute.  Transposing again provides the required map
\[
c(a,b) = \widetilde{\gamma(\tilde{a},\tilde{b})} : X\times B \to C.
\]
The converse is just as direct.
\end{proof}
%

\begin{proposition}\label{prop:intid} 
Given a natural model $p : \UU \to \U$,  the type theoretic rules for \emph{intensional} identity types are modelled by maps  
\begin{align}
\iy &: \UU \to \UU \label{prop:inteq1}\\
\Id &: \UU\times_\U \UU \to \U \label{prop:inteq2}
\end{align}
with $p\circ \iy = \Id\circ \delta$, and such that the canonical map $(\delta,\iy)$
\begin{equation*}
\xymatrix{
\UU \ar@/{}_{1pc}/[ddr]_{\delta} \ar@{.>}[dr]|-{(\delta,\iy)} \ar@/{}^{1pc}/[drr]^{\iy} &&\\
& I \ar[r]  \ar[d]  \pbcorner & \UU \ar[d]^p \\
& \UU\times_\U \UU \ar[r]_-{\Id} & \U
}
\end{equation*}
has a left-lifting structure $j$ with respect to $p$, when both are regarded as maps over $\U$,
$$(\delta,\iy)\ \pitchfork_j\ \U^*(p).$$

\end{proposition}

\begin{proof}
Let us write $\rho = (\delta,\iy)$.  By  lemma \ref{lem:extintLLP}, a left-lifting structure $j$ for $\rho$ with respect to $p$,  both  regarded as maps over $\U$,  is equivalent to a natural (in $X$) choice of diagonal fillers $j(a,b)$ for all squares over $\U$ of the form
\[
\xymatrix{
X\times_\U \UU \ar[d]_{X\times_\U{\rho}} \ar[rr]^{a} && \U^*\UU \ar[d]^{\U^*p} \\
X\times_\U I \ar[rr]_{b} \ar@{.>}[urr]|-{\ j(a,b)\ } && \U^*\U,
}
\]
where $\U^* : \Chat \to \Chat/\U$ is the base change.  Letting $X = (A : \Gamma \to \U)$ as an object over $\U$, and consulting \eqref{diag:big}, we see that:
\begin{align*}
\dom(X\times_\U \UU)\ &=\ \G\exdot A \\
\dom(X\times_\U I) \ &=\  \G\exdot A\exdot A\exdot\Id_A  \\
\dom(X\times_\U{\rho})\ &=\ \rho_A : \G\exdot A \to  \G\exdot A\exdot A\exdot\Id_A 
\end{align*}
Thus, transposing the above diagram to forget the base $\U$, we arrive at the equivalent filling problem
\[
\xymatrix{
\G\exdot A \ar[d]_{\rho_A} \ar[rr] && \UU \ar[d]^{p} \\
 \G\exdot A\exdot A\exdot\Id_A  \ar[rr]  \ar@{.>}[urr]|-{\ \jay\ }&& \U.
}
\]
Comparing this to the diagram \eqref{diag:Idfill}, we see that the assumed left-lifting structure $j$ indeed provides a choice of fillers $\jay$ that is natural in $\G$, as required to correctly interpret the elimination rule.
\end{proof}

\section{Supporting a natural model}

The representability of a natural transformation $p : \UU\to\U$ imposes conditions on the maps in $\C$ that represent it (cf.~corollary \ref{cor:repnattransstableclass}), and the requirement that $p$ should model the type-forming operations $\Sigma, \Pi, \Id$ imposes further conditions on those maps. Our goal is to determine conditions on a category $\C$ that are sufficient to  ensure that it carries a natural model of type theory.

Let $p : \UU\to\U$ be a representable natural transformation over $\C$.  Recalling our convention from Section \ref{subsection:algebraic}, for each object $C\in\C$ and each element $A\in\U(C)$, we have a selected object $\btexdot{C}{A}$, a map $p_{A}:\btexdot{C}{A} \to C$ and an element $q_A\in \UU(\btexdot{C}{A})$, all fitting into a pullback square of the form:
\begin{equation}\label{diag:prep}
\xymatrix{
\ext{C}{A} \ar[d]_{p_A} \ar[r]^{q_A} \pbcorner &  \UU \ar[d]^{p}\\
C \ar[r]_{A}   & \ \U\,.}
\end{equation}
Such a map $p_{A}:\btexdot{C}{A} \to C$ arising as a canonical pullback of $p$ will be called a \emph{display map}, and the corresponding pullback square, a \emph{display square for~$p_A$}. 

\begin{remark}\label{remark:displaypb}
Observe that a display map has a pullback along any map, even though $\C$ is not assumed to have all pullbacks.  
Indeed, for any display map $p_{A}:\btexdot{C}{A} \to C$ and any map $s : D\to C$, there is  a uniquely determined pullback square with $p_{As} : \ext{D}{As} \to D$ a display map, as shown on the left in the diagram below:
\[
\xymatrix{
\ext{D}{As} \ar@/{}^{2pc}/[rr]^{q_{As}} \ar[d]_{p_{As}} \pbcorner \ar@{.>}[r] & \ext{C}{A} \ar[d] \ar[d]_{p_A} \ar[r]^{q_A} \pbcorner &  \UU \ar[d]^{p}\\
D \ar@/{}_{2pc}/[rr]_{As} \ar[r]_{s}   & C \ar[r]_{A}   & \U}
\]
because the outer rectangle and the righthand square are canonical pullbacks.
\end{remark}

Conversely, let $\D\subseteq\C_1$ be a class of maps in $\C$ that is closed under all pullbacks along arbitrary maps in $\C$, in the sense that (i) the pullback of a map in \D\ along any map in \C\ always exists, and (ii) given any pullback square 
\begin{equation}
\xymatrix{
A\ar[d]_{g} \ar[r] \pbcorner &  C\ar[d]^{f}\\
B \ar[r]   & D}
\end{equation}
if $f\in\D$ then $g\in\D$.  Observe that $\D$ is closed under isomorphisms in the arrow category.  Call such a class $\D\subseteq\C_1$  a \emph{stable class of maps} in $\C$.  We define two presheaves $\D_0, \D_1$ and a natural transformation $\pi : \D_1\to\D_0$ between them as follows:
\begin{align*}
\D_1(C)\ &=\ \{ (a, d)\in \C_1\times \D\ |\ \cod(a) = \dom(d) \} \\
\D_0(C)\ &=\ \{ (b, d) \in  \C_1\times \D\ |\ \cod(b) = \cod(d) \}\\
\pi_C &: \D_1(C) \to \D_0(C) \\
& \pi_C(a,d)\ =\ (d\circ a, d).
\end{align*}
Schematically, we have the following situation.
\begin{align*}
\D_1(C)\ &=\ \left\{ \vcenter{ \xymatrix{& \cdot \ar[d]^{d\,\in\,\D} \\
				C\ar[ru]^{a} & \cdot} }
			\right \}	\\
		&\xymatrix{ \strut \ar[d]_{\pi_C}\\
		\strut } \\
\D_0(C)\ &=\ \left\{ \vcenter{ \xymatrix{& \cdot \ar[d]^{d\,\in\,\D} \\
				C \ar[r]_{b} & \cdot} }
			\right\}
\end{align*}
The action of the presheaves is by precomposition in the first factor, thus for $s : C' \to C$, we let
\begin{align*}
\D_1(s)(a, d)\ &=\ (a\circ s, d), \\
\D_0(s)(b, d)\ &=\ (b\circ s, d) .
\end{align*}
This is plainly (strictly) functorial.  The component $\pi_C$ is simply composition with the arrow $d$ in the second factor, 
which is obviously natural.

\begin{remark}\label{remark:Dalternate}
The natural transformation $\pi : \D_1\to\D_0$ can be defined explicitly by
\[
\coprod_{d\in\D}\y{d} : \coprod_{d\in\D}\y\dom(d) \to \coprod_{d\in\D}\y\cod(d)\,.
\]
The current description of $\pi: \D_1\to \D_0$ is more closely related to a coherence theorem for certain kinds of indexed categories (respectively fibrations), which takes the pseudofunctor $\D : \C\op\to\Cat$ given by a stable class of maps, with action by pullback, and returns an equivalent presheaf of categories, i.e.\ a ``strictification" of the pseudofunctor (or ``splitting" of the associated fibration). Several such strictifications have been studied previously: this one is left adjoint to the inclusion of functors into pseudofunctors, and there is also a right adjoint, and others (all three are attributed to Giraud in cf.~\cite{Streicher}, which also gives the relation to work of Benabou). The use of this left adjoint construction to obtain a model of intensional type theory is the main result of \cite{LW}, and the development in this section can be regarded as a reformulation, to the present setting of natural models, of results obtained in \cite{LW} for the closely related setting of categories with attributes. Also see remark \ref{rem:credit} below.
\end{remark}

\begin{proposition}\label{prop:Dtorepnattrans}
Let $\D\subseteq\C_1$ be a stable class of maps in $\C$.  Then the natural transformation  $\pi : \D_1\to \D_0$ just defined is representable.
\end{proposition}

\begin{proof}
Let $C\in\C$ and $A\in\D_0(C)$.  We require an object $\btexdot{C}{A}$, a map $p_{A}:\btexdot{C}{A} \to C$ and an element $q_A\in \D_1(\btexdot{C}{A})$ fitting into a pullback square of the form:
\begin{equation}\label{diag:prep2}
\xymatrix{
\ext{C}{A} \ar[d]_{p_A} \ar[r]^{q_A} \pbcorner &  \D_1\ar[d]^{\pi}\\
C \ar[r]_{A}   & \,\D_0\,.}
\end{equation}
Now $A\in\D_0(C)$ is a cospan of the form, say,
\[
\xymatrix{& A_1 \ar[d]^{d_A} \\
				C \ar[r]_{|A|} & A_0}
				\]
with $d_A \in \D$.  So we can take a pullback to define $p_{A}:\btexdot{C}{A} \to C$ and $q'_A $  as indicated in:
\begin{equation}\label{diag:prep3}
\xymatrix{
\ext{C}{A} \ar[d]_{p_A} \ar[r]^{q'_A} \pbcorner &  A_1\ar[d]^{d_A}\\
C \ar[r]_{|A|}   & A_0\,.}
\end{equation}
Let $q_A \in \D_1(\btexdot{C}{A})$ be defined by  $q_A = (q'_A, d_A)$. To see that the square \eqref{diag:prep2} commutes, observe that 
\[
\pi_{\btexdot{C\,}{\,A}}(q'_A, d_A) = (d_A\circ q'_A, d_A) = (|A|\circ p_A, d_A) = \D_0(p_{A})(|A|, d_A) = \D_0(p_{A})(A).
\]
The proof that \eqref{diag:prep2} is a pullback is a routine unwinding of the definitions.
\end{proof}

\begin{corollary}\label{cor:repnattransstableclass}
A representable natural transformation determines a stable class of maps $\D\subseteq \C_1$, namely all those maps isomorphic to display maps, and every stable class of maps $\D\subseteq \C_1$ is determined in this way by a representable natural transformation $\pi : \D_1\to\D_0$.
\end{corollary}

\begin{proof}
Every display map is clearly in \D, and a display map has a pullback along any map by Remark \ref{remark:displaypb}.  It follows directly that $\D$ is stable.  Conversely, every map $d : D\to C$ in \D\ occurs as a display map for the associated representable natural transformation $\pi : \D_1\to\D_0$ of Proposition \ref{prop:Dtorepnattrans}, in the form: 
\[
\xymatrix{
\ext{C}{D} \ar[d]_{p_D} \pbcorner \ar[r]^{\identity} &  D\ar[d]^{d}\\
C \ar[r]_{\identity}   & C\,.}
\]
\end{proof}

\begin{remark}
Note that by specifying $\pi : \D_1\to\D_0$ as $\pi = \coprod_{d\in\D}\y{d}$, as mentioned in Remark \ref{remark:Dalternate}, we can obtain a simplification of Proposition \ref{prop:Dtorepnattrans} and its corollary: the representability of $\pi$ follows immediately from the indecomposibility of representable functors and the good behavior of coproducts of presheaves. The corollary then also follows more directly.
But also note that different representable natural transformations on a category $\C$ may give rise to the same stable class of maps $\D\subseteq\C_1$.  We shall not pursue this line of inquiry further, since it is not required for what follows.
\end{remark}

Our task now is to determine conditions on a stable class of maps \D\ that will ensure that the associated representable natural transformation $\pi : \D_1 \to \D_0$ models the various type-theoretic rules in the sense  determined in section \ref{sec:modelling}.

\subsection{Sums and Products}

Recall from proposition \ref{prop:sum} the condition on $\pi : \D_1 \to \D_0$ required to model the rules for sum types $\Sigma$:
there should be maps  
\begin{align*}
\pairmap :& \sum_{(A\ty\D_0)}\sum_{(B\ty\D_0^{\cors{A}})}\sum_{(a : \cors{A})}B(a) \to \D_1\\
\Sigma :& \sum_{A\ty\D_0}\D_0^{\cors{A}} \to \D_0\label{eq:sumop2}
\end{align*}
making the following diagram a pullback,
\begin{equation}\label{diag:sigmapb}
\xymatrix{
 \sum_{(A\ty\D_0)}\sum_{(B\ty\D_0^{\cors{A}})}\sum_{(a : \cors{A})}B(a) \ar[d]_{\pi_1} \ar[rr]^-{\pairmap} &&  \D_1 \ar[d]^{\pi}\\
\sum_{A\ty\D_0}\D_0^{\cors{A}} \ar[rr]_-{\Sigma} && \D_0 }
\end{equation}
where $\cors{A}$ denotes the fiber of $\pi$ over $A$, i.e.\ the object given by pullback:
\begin{equation*}
\xymatrix{
\cors{A} \ar[d] \ar[r] \pbcorner &  \D_1 \ar[d]^{\pi}\\
X \ar[r]_A & \D_0. }
\end{equation*}

Take any $X\in\C$ and $(A,B) : X \to \sum_{A\ty\D_0}\D_0^{\cors{A}}$, and we seek an assignment of a map $\Sigma(A,B) : X \to \D_0$, in a way that is natural in $X$.  

Using Lemma \ref{lem:polyclass}, the map $(A,B) : X \to \sum_{A\ty\D_0}\D_0^{\cors{A}}$ uniquely determines maps $A : X \to \D_0$ and $B : \cors{A} = X\times_{\D_0} \D_1 \to \D_0$, as already suggested by the notation. These in turn correspond uniquely (by Yoneda) to cospans:
\begin{align*}
A\ &= (a\in\C, p\in\D) \\
B\ &= (b\in\C, q\in\D)
\end{align*}
as indicated in:
\begin{equation}\label{diag:wrongsigma}
\xymatrix{
& B_1 \ar[dd]^{q} &\\
&& \\
& B_0 &\\
X\times_{A_0} A_1 \ar[ru]_-{b} \ar[d]_{p_X} \ar[rr] && A_1 \ar[d]^{p}\\
X \ar[rr]_{a} && A_0.}
\end{equation}
Here we have used the following easily proved fact, which we record for later reuse:

\begin{lemma}\label{lem:reindex}
When $A : X\to\D_0$ corresponds to the cospan $(a,p)$, then $A$ factors through $a : X \to A_0$ via the map $(\identity_{A_0}, p) : A_0\to\D_0$, and the following is then a pullback:
\[
\xymatrix{
A_1 \ar[d]_{p} \ar[rr] \pbcorner &&  \D_1 \ar[d]^{\pi}\\
A_0 \ar[rr]_{(\identity,\,p)} && \D_0 .}
\]
Thus
\[
\cors{A}\ =\ X\times_{\D_0} \D_1\ =\ X\times_{A_0} A_1\,.
\]
\end{lemma}

Returning to diagram \eqref{diag:wrongsigma}, it might now be expected that the sum $\Sigma(A,B)$ would be built by first pulling $q$ back along $b$ to give $q_X$, and then composing with $p_X$:

\begin{equation}\label{diag:wrongsigma2}
\xymatrix{
& B_1 \ar[dd]^{q} &\\
\cdot \ar@/_8ex /[ddd]_{\Sigma(A,B)?} \ar[ru] \ar[dd]^{q_X} && \\
& B_0 &\\
X\times_{A_0} A_1 \ar[ru]_-{b} \ar[d]^{p_X} \ar[rr] && A_1 \ar[d]^{p}\\
X \ar[rr]_{a} && A_0.}
\end{equation}
This is ``morally" what we want to do, since the resulting composite is indeed the display map  $\btexdot{X}{\Sigma(A,B)}\to X$, and so the requirement that  \emph{$\D$ is closed under composition} suggests itself.  There is a problem with this construction, however: $\Sigma(A,B)$ must be a cospan $(c\in\C,d\in\D)$ of the form:  $\xymatrix{& \cdot \ar[d]^{d} \\
				X \ar[r]_c & \cdot}$
But the only candidate in sight for $c$ is the identity on $X$, and that assignment would not be natural in $X$!  

Instead, we shall use a construction similar to that applied in section \ref{sec:modelling} to devise a ``generic case" in which to perform the operation of pullback-plus-composition, so that all other cases result simply from mapping into the generic one.  This construction, however, requires that \D\  not only be closed under composition, but also that certain \emph{right} adjoints to pullback exist.  To state the required condition precisely, for any object $C\in\C$, let us write $\D(C)$ for the full subcategory $\D(C)\hook\C/C$ on the \D-maps into $C$ as objects.  

\begin{definition}\label{def:closed}
A stable class of maps $\D\subseteq\C_1$ is \emph{closed} if the following conditions hold:
\begin{enumerate}

\item \C\ has a terminal object $1$, and every map $C\to 1$ is in \D.

\item \D\ is closed under composition.

\item  For any $d : D\to C$ in \D, the pullback functor $d^* : \D(C) \to \D(D)$ has a right adjoint $d_* : \D(D) \to \D(C)$, and the inclusion functor $\D(C)\hook\C/C$ preserves exponentials.
\end{enumerate}
\end{definition}

\begin{proposition}\label{prop:sumprod}
If $\D\subseteq\C_1$ is a closed, stable class of maps, then the associated representable natural transformation $\pi:\D_1\to\D_0$ models the rules for sums $\Sigma$ and products $\Pi$.
\end{proposition}

\begin{proof}
Taking up the argument from diagram \eqref{diag:wrongsigma}, consider the following construction:  
\begin{equation}\label{genericcase}
\xymatrix@=3em{
& B_1 \ar[dd] |\hole ^>>>>>>>{q} & &\\
\cdot \ar[ru] \ar[dd]_{q_X} \ar[rr] && \ar[lu] G \ar[dd]^{q'} & \\
& B_0 & &\\
X\times_{A_0} A_1\ar[ru]^{b} \ar[d]_{p_X}  \ar[rr]_{\overline{b}\times_{A_0}A_1} && \ar[lu]_{\mathrm{ev}} B_{0}^{A_1} \times_{A_0} A_1 \ar[d]^{p'} \ar[r] & A_1 \ar[d]^{p}\\
X \ar[rr]^{\overline{b}} \ar@/_5ex /[rrr]_{a} && B_0^{A_1} \ar[r]  & A_0.}
\end{equation}
We first factor the map $a: X\to A_0$ through the transpose $\overline{b} : X \to B_{0}^{A_1}$ of $b : X\times_{A_0} A_1\to B_0$ over $A_0$ (regarding $B_0$ as a constant object over $A_0$ by base change along $A_0 \to 1$).  Here we know that $B_{0}^{A_1}$ exists in $\D(A_0)$ because both $p$ and $B_0\to 1$ are  in \D, and we know that $B_{0}^{A_1}$ is also an exponential in $\C/A_0$ by the definition of ``closed".

Pulling $p$ back along $a$ in two stages gives the two lower pullback squares.  Next, still working over $A_0$, the map $b$ now factors as $\mathrm{ev}\circ(\overline{b}\times_{A_0}A_1)$ by the exponential adjunction. 
The pullback $q_X$ of $q$ along $b$ can therefore also be constructed in two stages, giving first the map $q' : G \to B_{0}^{A_1} \times_{A_0} A_1$ as the pullback of $q$ along the evaluation $\mathrm{ev}$. 

The generic case of the  ``pullback and compose" construction \eqref{diag:wrongsigma2} that we seek now has the form:
\[
\xymatrix{
B_1 \ar[d] ^{q} & &\\
B_0 &  \ar[lu]  \ar@/_8ex/ [dd] G \ar[d]^{q'} & \\
& \ar[lu] |<<<<<<\hole B_{0}^{A_1} \times_{A_0} A_1 \ar[d]^{p'} \ar[r] & A_1 \ar[d]^{p}\\
& B_0^{A_1} \ar[r]  & A_0.}
\]
The composite $p'\circ q' : G \to B_0^{A_1}$ is the $\D$ component of the desired cospan $(\overline{b},\, p'\circ q')$ defining $\Sigma(A,B) : X\to \D_0$.  Observe that the pullback of $p'\circ q'$ along $\overline{b}$ is indeed $p_X\circ q_X$, and that the same is true for any given $y : Y\to B_0^{A_1}$, because any such map is uniquely of the form $y=\overline{c} : Y\to B_0^{A_1}$ for $c = \mathrm{ev}\circ(y\times_{A_0}A_1) : Y\times_{A_0} A_1 \to B_0$.  

This defines the natural transformation $\Sigma : \sum_{A\ty\D_0}\D_0^{\cors{A}} \to \D_0$.  
Explicitly, given $(A, B): X \to \sum_{A\ty\D_0}\D_0^{\cors{A}}$, where  $A = (a, p)$ and $B = (b, q)$, we define $\Sigma(A,B) : X\to \D_0$ by $\Sigma(A,B) = (\overline{b},\, p'\circ q')$.  This assignment is natural in $X$, for given any $s: Y\to X$, we have 
$$\Sigma(A,B)s = (\overline{b}\circ s,\, p'\circ q') = (\overline{b\circ s'},\, p'\circ q') = \Sigma(As,Bs),$$
because the \D-component is fixed, and exponential transposition is natural.
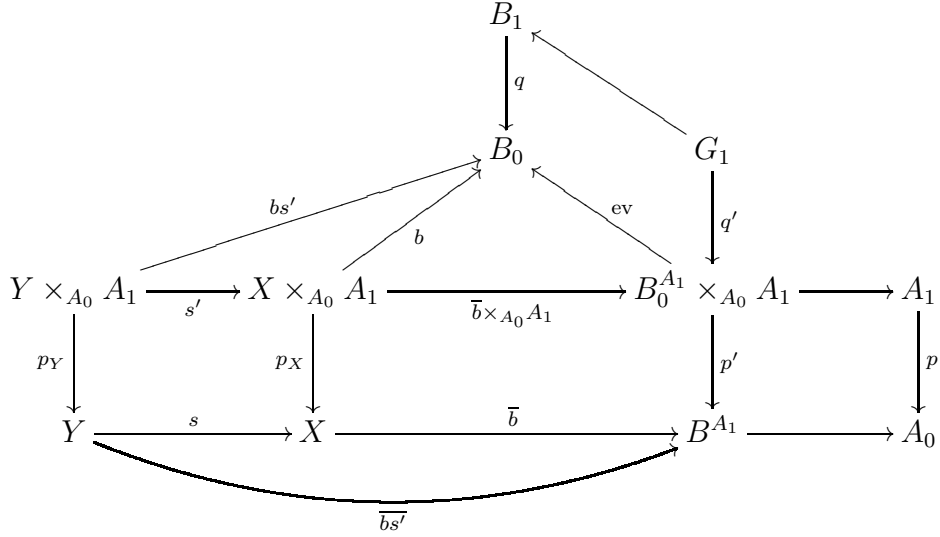
\[
\xymatrix@=3em{
&& B_1 \ar[d] ^{q} & &\\
&&B_0 & \ar[lu] G_1 \ar[d]^{q'} & \\
Y\times_{A_0} A_1\ar[rru]^{bs'} \ar[d]_{p_Y}  \ar[r]_{s'} & X\times_{A_0} A_1\ar[ru]_{b} \ar[d]_{p_X}  \ar[rr]_{\overline{b}\times_{A_0}A_1} && \ar[lu]_{\mathrm{ev}} B_{0}^{A_1} \times_{A_0} A_1 \ar[d]^{p'} \ar[r] & A_1 \ar[d]^{p}\\
Y \ar[r]^{s} \ar@/_5ex /[rrr]_{\overline{bs'}} & X \ar[rr]^{\overline{b}} && B^{A_1} \ar[r]  & A_0}
\]

To define the pairing map,
\[
\pairmap : \sum_{(A\ty\D_0)}\sum_{(B\ty\D_0^{\cors{A}})}\sum_{(a : \cors{A})}B(a) \to \D_1,
\]
take an element $(A,B,c): X\to \sum_{(A\ty\D_0)}\sum_{(B\ty\D_0^{\cors{A}})}\sum_{(a : \cors{A})}B(a)$, and we require an element  $\pairmap (A,B,c) : X\to \D_1$ via an assignment that is natural in $X$.  The map $(A,B,c)$ determines data of the form:
\begin{align*}
A\ &= (a\in\C, p\in\D) \\
B\ &= (b\in\C, q\in\D)\\
c &= (a',b'),
\end{align*}
where:
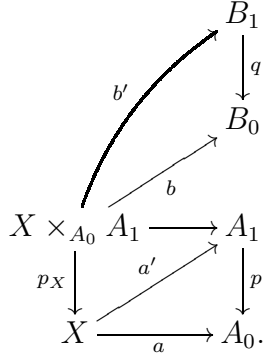
\begin{equation*}
\xymatrix{
& B_1 \ar[d]^{q} &\\
& B_0 &\\
X\times_{A_0} A_1 \ar@/^2ex /[ruu]^{b'} \ar[ru]_-{b} \ar[d]_{p_X} \ar[r] & A_1 \ar[d]^{p}\\
X \ar[r]_{a} \ar[ru]^{a'} & A_0.}
\end{equation*}

But this is just a section of the composite $p_X\circ q_X$,
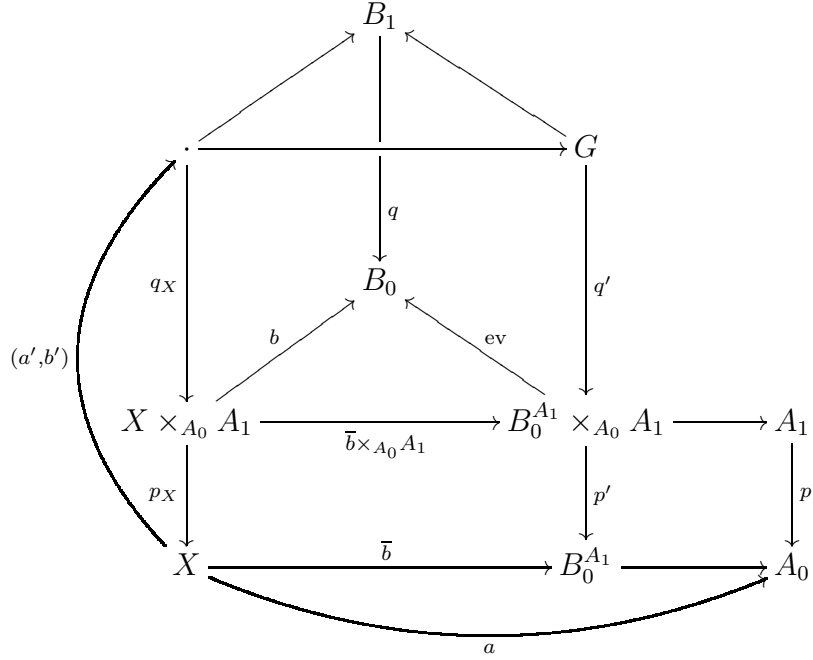
\[
\xymatrix@=3em{
& B_1 \ar[dd] |\hole ^>>>>>>>{q} & &\\
\cdot \ar[ru] \ar[dd]_{q_X} \ar[rr] && \ar[lu] G \ar[dd]^{q'} & \\
& B_0 & &\\
X\times_{A_0} A_1\ar[ru]^{b} \ar[d]_{p_X}  \ar[rr]_{\overline{b}\times_{A_0}A_1} && \ar[lu]_{\mathrm{ev}} B_{0}^{A_1} \times_{A_0} A_1 \ar[d]^{p'} \ar[r] & A_1 \ar[d]^{p}\\
X \ar@/^8ex /[uuu]^{(a',b')} \ar[rr]^{\overline{b}} \ar@/_5ex /[rrr]_{a} && B_0^{A_1} \ar[r]  & A_0}
\]
or, equivalently, a section of the composite $p'\circ q'$ over $\overline{b}$.  But this in turn is exactly an element $c'$ of the generic $\Sigma$-type,
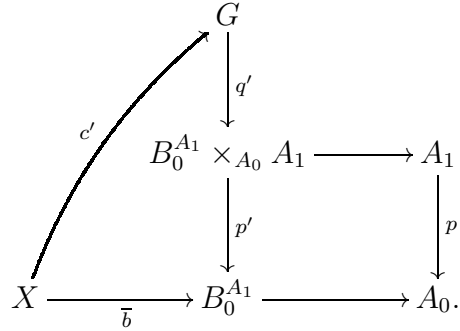
\[
\xymatrix@=3em{
&  G \ar[d]^{q'} & \\
& B_{0}^{A_1} \times_{A_0} A_1 \ar[d]^{p'} \ar[r] & A_1 \ar[d]^{p}\\
X \ar@/^2ex /[ruu]^{c'} \ar[r]_{\overline{b}} & B_0^{A_1} \ar[r]  & A_0.}
\]
So we can set $$\pairmap_X(A,B,c) = (c',p'\circ q')\in \D_1(X).$$
Again, this is plainly natural in $X$, because the action in the first component is precomposition and second component is fixed.

It is  immediate that this assignment makes \eqref{diag:sigmapb} a pullback: for fixed $A = (a, p)$ and  $B = (b, q)$, the correspondence $c=(a',b') \mapsto c'$ is clearly reversible.

For the products $\Pi$, we start from the object constructed in  \eqref{genericcase}:
\begin{equation*}
\xymatrix{
& B_1 \ar[d]_{q} & &\\
& B_0 &  \ar[lu] G \ar[d]^{q'} & \\
&& \ar[lu]_{\mathrm{ev}} B_{0}^{A_1} \times_{A_0} A_1 \ar[d]^{p'} \ar[r] & A_1 \ar[d]^{p}\\
&& B_0^{A_1} \ar[r]  & A_0.}
\end{equation*}
But now rather than composing $p'\circ q'$, we use the \emph{right} adjoint $p'_*$ to pullback along $p'$ to build the map $p'_*q' : G' \to B_0^{A_1} $:
\begin{equation}\label{genericcase2}
\xymatrix{
& B_1 \ar[d]_{q} & &\\
& B_0 &  \ar[lu] G \ar[d]^{q'} & \\
&G' \ar[rd]_{p'_*q'} & \ar[lu]_{\mathrm{ev}} B_{0}^{A_1} \times_{A_0} A_1 \ar[d]^{p'} \ar[r] & A_1 \ar[d]^{p}\\
&& B_0^{A_1} \ar[r]  & A_0.}
\end{equation}
Note that $p'_*q'$ (exists and) is in \D\ by our assumption that \D\ is closed.

Now, as in the previous case, given $(A,B) : X \to \sum_{A\ty\D_0}\D_0^{\cors{A}}$,  we have $A = (a\in\C, p\in\D)$ and $B = (b\in\C, q\in\D)$, from which we can construct $p'$, $q'$, and $\overline{b}: X\to B_0^{A_1}$. Then set:
\[
\Pi(A,B)\ =\ (\overline{b},\, p'_*q') : X \to \D_0\,.
\] 
The assignment is again obviously natural in $X$. 
The construction of  $\lambda$ and  verification that the resulting square is a pullback are entirely analogous to the case of $\pairmap$, and are omitted.

Finally, observe that for any $y : Y\to B_0^{A_1}$, the Beck-Chevalley conditions for the left and right adjoints to pullback $y^*$ give:
\begin{align*}
y^*(p'\circ q') &= (p_Y)\circ(q_Y)\\
y^*(p'_*q') &= (p_Y)_*(q_Y)\,.
\end{align*}
This ensures that the context extension operation behaves correctly.
\end{proof}

\subsection{Identity types}

As was the case for sums and products, in order to model intensional identity types $\Id$, we require an additional condition on the stable class of maps $\D\subseteq\C_1$.  It may be surprising that we also still need the class to be \emph{closed} in the sense of definition \ref{def:closed}; this is used to again construct certain ``generic" cases.  

Let $\D\subseteq\C_1$ be a class of maps in a category \C.
We shall say that a map $a : A \to B$ in \C\ is  \emph{anodyne} if it has the left lifting property with respect to all maps in $\D$.  The class $\D\subseteq\C_1$ will be called \emph{factorizing} if every map $f : A\to B$ in \C\ factors as $f = d\circ a$ with  $a$ anodyne and  $d\in\D$,
\[
\xymatrix{
& B' \ar[d]^{d}\\
A \ar[ru]^{a} \ar[r]_{f} & B.}
\]
%
%

\begin{lemma}\label{lemma:anodynestable}
If $\D \subseteq \C_1$ is a closed, stable, factorizing class of maps, then the anodyne maps are preserved by pullback along all maps in \D. Moreover, any pullback of an anodyne map between two \D-maps is again anodyne.
\end{lemma}
\begin{proof}
This is a familiar fact in axiomatic homotopy theory (also cf.~proposition 14 of \cite{GG}).  Briefly, let $d : D \to C$ in \D\ and $a :A\to C$ anodyne, and consider $d^*a : d^*A \to D$. This is also anodyne if it lifts against any $b : B \to D$ in \D\ (using the fact that \D-maps are preserved under base change).  Applying $d_*$ to $b$ gives a corresponding lifting problem at $C$ involving $a$ and $d_*b$, which has a solution since $a$ is anodyne and $d_*b$ is in \D.  Transposing the lift across the adjunction $d^* \dashv d_*$ gives the solution over $D$.

For the second statement, suppose given $f : A \to Y$ and $g : B \to Y$ in \D\ and anodyne $a : A \to B$ over $Y$. Pull back along any $h : X \to Y$ to get $f' : A' \to X$ and $g' : B' \to X$ in \D\ with $a' : A' \to B'$ over $X$, which we claim is also anodyne. \[ \xymatrix{ A' \ar[dd]_{f'} \ar[rd]^{a'} \ar[rrr]^{f^*h} &&& A \ar[dd]_(.3){f} \ar[rd]^{a} \\ & B' \ar[ld]^{g'} \ar[rrr]_(.3){g^*h} &&& B \ar[ld]^{g} & \\ X \ar[rrr]_{h} &&& Y .} \] Since anodyne maps are preserved by pullback along \D-maps, it suffices to assume that $h$ is anodyne (else factor it into an anodyne followed by a \D-map). The pullback $f^*h : A' \to A$ of $h$ along $f$ is then anodyne, and so is $g^*h : B' \to B$. Since $a \circ f^*h = g^*h \circ a'$, and anodyne maps are closed under composition, we shall have $a'$ anodyne once we prove the following: Given any maps $C \stackrel{i}{\to} D \stackrel{j}{\to} E$, if both $j$ and $j \circ i$ are anodyne, so is $i$. \[ \xymatrix{ C \ar[d]_{i} \ar[rr]^{c} && F \ar[d]^{f} \\ D \ar[d]_{j} \ar[rr]_(.6){1_D} \ar@{.>}[rru]^d && D \\ E \ar@{.>}[rru]_{e} \ar@{.>}[rruu]^(.3){k} &&.} \] To prove this, we take any $f : F \to D$ in \D\ and $c : C \to F$ with $f\circ c = i$, and produce a diagonal filler $d : D \to F$, with $d\circ i = c$ and $f\circ d = 1_D$. Since $j$ is anodyne and $D$ is fibrant (all objects in \C\ are fibrant by the definition of ``closed''), there is an $e : E \to D$ with $e\circ j = 1_D$. Since $j\circ i$ is anodyne and $f$ in \D\ there is a $k : E \to F$ with $k\circ j\circ i = c$ and $f\circ k = e$. Let $d = k\circ j$. Then $d\circ i = k\circ j\circ i = c$, and $f\circ d = f\circ k\circ j = e\circ j = 1_D$, as required.
\end{proof}

\begin{proposition}\label{prop:id}
If $\D \subseteq \C_1$ is a closed,  stable, factorizing class of maps, then the associated representable natural transformation $\pi:\D_1\to\D_0$ models the rules for intensional identity types $\Id$.
\end{proposition}
\begin{proof}
Recall from proposition \ref{prop:intid} that we require maps  
\begin{align}
\iy &: \D_1 \to \D_1\\
\Id &: \D_1\times_{\D_0} \D_1 \to \D_0
\end{align}
commuting with $\pi$ and its diagonal $\delta$, 
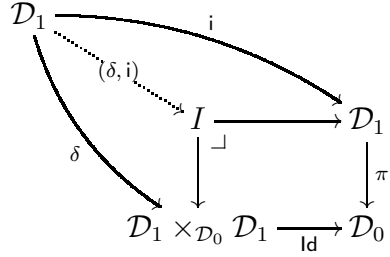
\begin{equation*}
\xymatrix{
\D_1 \ar@/{}_{1pc}/[ddr]_{\delta} \ar@{.>}[dr]|-{(\delta,\,\iy)} \ar@/{}^{1pc}/[drr]^{\iy} &&\\
& I \ar[r]  \ar[d]  \pbcorner & \D_1 \ar[d]^{\pi} \\
& \D_1\times_{\D_0} \D_1 \ar[r]_-{\Id} & \D_0
}
\end{equation*}
and a left-lifting structure $j$ for the map $(\delta,\iy)$ with respect to $\pi$,
$$(\delta,\iy)\ \pitchfork_j\ \pi$$
where both are regarded as maps over $\D_0$.

Again, we shall write $\rho = (\delta,\iy) : \D_1 \to I$.

We begin by constructing the map 
\[
\Id : \D_1\times_{\D_0} \D_1 \to  \D_0.
\]
For each $A\in\C$, pick a factorization of the diagonal,
\[
\xymatrix{
& I_A \ar[d]^{d_A}\\
A \ar[ru]^{r_A} \ar[r]_-{\delta_A} & A\times A}
\]
with $r_A$ anodyne and $d_A\in\D$, and do the same for every map $A: A_1 \to A_0$ in \D,
\[
\xymatrix{
& I_A \ar[d]^{d_A}\\
A_1 \ar[ru]^{r_A} \ar[r]_-{\delta_A} & A_1\times_{A_0} A_1.}.
\]
(Of course, the second step subsumes the first.)

For $X\in\C$, a map $\alpha: X\to  \D_1\times_{\D_0} \D_1$ consists of a map $A : X \to {\D_0}$ together with two maps $a_1,a_2 : X\to \D_1$ over ${\D_0}$.  Now $A$ is a cospan $A = (a\in \C,p\in\D)$, and there is a  pullback diagram,
\[
\xymatrix{
&& A_1 \ar[d]^{p} \ar[r] \pbcorner & \D_1 \ar[d]^{\pi}\\
X  \ar@<1ex>[urr]^{a_1} \ar[urr]_{a_2}  \ar[rr]_{a} \ar@/_4ex /[rrr]_{A} && A_0 \ar[r] & \D_0}
\]
with the corresponding elements $a_1, a_2$ fitting in as shown.  These in turn determine an element $(a_1,a_2)$ of $A_1\times_{A_0} A_1$, which we could also have constructed directly, as indicated in the following:
\[
\xymatrix{
X  \ar@/^6ex /[rrr]^{\alpha} \ar[rr]_-{(a_1,\, a_2)} \ar@/_2ex/[rrdd]_{a} && \ar@<-1ex>[d] \ar@<.5ex>[d]  A_1\times_{A_0} A_1 \ar[r] & \ar@<-1ex>[d] \ar@<.5ex>[d]  \D_1\times_{\D_0} \D_1\\
&& A_1 \ar[d]^{p} \ar[r] \pbcorner & \D_1 \ar[d]^{\pi} \\
&& A_0 \ar[r] & \D_0 .}
\]
We require an element $\Id(\alpha) : X\to \D_0$, by an assignment that is natural in $X$.  For this, we take the following cospan:
\[
\xymatrix{
&& I_A \ar[d]^{d_A}\\
X  \ar[rr]_-{(a_1,\, a_2)} && A_1\times_{A_0} A_1 .}
\]

%


To define $\iy : \D_1\to \D_1$, an element of $\D_1(X)$ has the form $(a,p)$ with $a: X \to A_1$ and $p: A_1\to A_0$ with $p\in\D$.  Compose with $r_A : A_1 \to I_A$ to get 
\[
\iy(a,p)\ =\ (r_A\circ a, d_A),
\]
which is again an element of $\D_1(X)$:
\begin{equation}\label{diag:rho}
\xymatrix{
&& A_1 \ar[d]^{r_A} \\
X  \ar[rru]^{a} \ar[rr]_{r_A\circ\,a} \ar@/_2ex/[rrd]_{(a,a)} && I_A \ar[d]^{d_A}\\
&& A_1\times_{A_0} A_1.}
\end{equation}
This specification plainly makes $\pi\circ\iy(a,p)=\Id\circ\delta(a,p)$, as required.

Next, the presheaf $I:\C\op\to \Set$ has as elements of $I(X)$ pairs $$\alpha: X\to  \D_1\times_{\D_0} \D_1,\quad \beta : X\to \D_1$$ fitting together as follows:
\[
\xymatrix{
&& I_A \ar[d]^{d_A}\\
X \ar@/^2ex /[urr]^{b} \ar[rr]_-{(a_1,a_2)} \ar@/_2ex /[rrdd]_{a} && \ar@<-1ex>[d] \ar@<.5ex>[d]  A_1\times_{A_0} A_1\\
&& A_1 \ar[d]^{p}\\
&& A_0 }
\]
where $\alpha = (a,p)$ and $\beta=(b,d_A)$. 
The maps $\pi_1: I\to  \D_1\times_{\D_0} \D_1$ and $\pi_2:I\to \D_1$ are of course the projections.

Finally, the map $\rho = (\delta,\iy) : \D_1 \to I$ takes $(a,p): X\to\D_1$ with $a: X \to A_1$ and $p: A_1\to A_0$ to the pair:
\[
\rho(a,p) = (\delta(A,p), \iy(a,p)),
\]
which is indeed in $I(X)$ by diagram \eqref{diag:rho}.

Now by lemma \ref{lem:extintLLP}, a left-lifting structure $j$ for $\rho$ with respect to $\pi$ over $\D_0$ is equivalent to a natural (in $X$) choice of diagonal fillers $j(\alpha,\beta)$ for all squares over $\D_0$ of the form
\begin{equation}\label{jnat}
\xymatrix{
X\times_{\D_0} \D_1 \ar[d]_{X\times_{\D_0}{\rho}} \ar[rr]^{\alpha} && {\D_0}^*\D_1 \ar[d]^{{\D_0}^*\pi} \\
X\times_{\D_0} I \ar[rr]_{\beta} \ar@{.>}[urr]|-{\ j(\alpha,\beta)\ } && {\D_0}^*{\D_0}
}
\end{equation}
where ${\D_0}^* : \Chat \to \Chat/{\D_0}$ is the base change.  Let the object $X$ over ${\D_0}$ be $A : X \to {\D_0}$, with $X\in\C$ representable, which clearly suffices by naturality.  Using lemma \ref{lem:reindex}, there is a corresponding cospan $(a,p)$ and a double pullback diagram:
\[
\xymatrix{
X\times_{A_0}A_1 \ar[d]_{p_A} \ar[r] & A_1 \ar[d]^{p} \ar[r] \pbcorner & \D_1 \ar[d]^{\pi}\\
X \ar[r]_{a} \ar@/_4ex /[rr]_{A} & A_0 \ar[r] & \D_0 .}
\]
Thus $X\times_{\D_0} \D_1 = X\times_{A_0}A_1$ in diagram \eqref{jnat}.  Proceding similarly for the other expressions there, we have:
\begin{align*}
\dom(X\times_{\D_0} \D_1)\ &=\ X\times_{A_0}A_1 \\
\dom(X\times_{\D_0} I) \ &=\  X\times_{A_0} I_A  \\
\dom(X\times_{\D_0}{\rho})\ &=\ X\times_{A_0}r_A
\end{align*}
as displayed in the following diagram.
\begin{equation}\label{diag:big2}
\xymatrix@=3em{
 {X}\times_{A_0} A_1 \ar[d]_{{X\times_{A_0}r_A}} \ar[r]  & A_1 \ar[d]_{r_A} \ar[r] \pbcorner & \D_1 \ar[d]_{\rho} \ar[dr]^{\iy} & \\
 {X}\times_{A_0} I_A \ar[d] \ar[r]  & I_A \ar[d] \ar[r]	\pbcorner	& I \ar[d] \ar[r] 	\pbcorner				& \D_1 \ar[d]^{\pi} \\
 {{X}\times_{A_0} A_1\times_{A_0} A_1} \ar@<-1ex>[d] \ar@<.5ex>[d] \ar[r] & A_1\times_{A_0}A_1 \ar@<-1ex>[d] \ar@<.5ex>[d] 	 \ar[r]		& \D_1\times_{\D_0} \D_1 \ar[r]_-{\Id} \ar@<-1ex>[d] \ar@<.5ex>[d] 	& {\D_0} \\
 {X}\times_{A_0}{A_1} \ar[d]_{p_A} \ar[r] 	& A_1\ar[d]_{p}\ar[r]	\pbcorner	& \D_1 \ar[d]^{\pi} 						& \\
 {X} \ar[r]_a 			& A_0	\ar[r]					& {\D_0} .	&
}
\end{equation}
Transposing diagram \eqref{jnat} to forget the base ${\D_0}$,  we arrive at the equivalent filling problem
\[
\xymatrix{
X\times_{A_0}A_1 \ar[d]_{{X\times_{A_0}r_A}} \ar[rr]^{\alpha} && \D_1 \ar[d]^{\pi} \\
X\times_{A_0} I_A   \ar[rr]_{\beta}  \ar@{.>}[urr] && \D_0
}
\]
to be solved naturally in $X$.  Now $\beta$ is a cospan of the form:
\[
\xymatrix{
&& B_1 \ar[d]^{q} \\
X\times_{A_0} I_A  \ar[rr]_-{b} && B_0
}
\]
with $q\in\D$.  And $\alpha = (c, q)$ completes the square,
\begin{equation}\label{diag:pause}
\xymatrix{
X\times_{A_0}A_1 \ar[d]_{{X\times_{A_0}r_A}} \ar[rr]^-{c} && B_1 \ar[d]^{q} \\
X\times_{A_0} I_A  \ar[rr]_-{b} && B_0.
}
\end{equation}
By lemma \ref{lemma:anodynestable}, ${X\times_{A_0}r_A}$ is anodyne,  and $q$ is in \D\ by assumption, so there is a diagonal filler $j(c,b)$ for this case, but we need to make a systematic choice that will be natural in $X$.  In order to do this, we will again construct a generic case from which all others arise by mapping in.  For that, we require the following.

\begin{lemma}
Given maps $p:A_1\to A_0$ and $q:B_1\to B_0$ in a  locally cartesian closed category, there is an object $G$ with maps $e_1:G\times A_1 \to B_1$ and $e_0 : G\times A_0\to B_0$ such that $e_0 \circ (G\times p) = q\circ e_1$:
\[
\xymatrix{
G\times A_1 \ar[d]_{G\times p} \ar[r]^-{e_1} & B_1 \ar[d]^{q} \\
G\times A_0   \ar[r]_-{e_0} & B_0
}
\]
and such that, given any object $X$ with maps $f_1:X\times A_1 \to B_1$ and $f_0 : X\times A_0\to B_0$ such that $f_0 \circ (X\times p) = q\circ f_1$, there is a (unique) map $$f : X\to G$$ such that $f_i = e_i\circ(f\times A_i)$ for $i=0,1$:
\[
\xymatrix{
X\times A_1 \ar@/^4ex /[rrr]^{f_1} \ar[d]_{X\times p} \ar[rr]_{f\times A_1} && G\times A_1 \ar[d]_{G\times p} \ar[r]_-{e_1} & B_1 \ar[d]^{q} \\
X\times A_0   \ar@/_4ex /[rrr]_{f_0} \ar[rr]^{f\times A_0} && G\times A_0   \ar[r]^-{e_0} & B_0\,.
}
\]
In other words, $(G,e_0,e_1)$ is a universal object for the presheaf (in $X$) of commutative diagrams of the form  
\[
\xymatrix{
X\times A_1 \ar[d]_{X\times p} \ar[r]^-{f_1} & B_1 \ar[d]^{q} \\
X\times A_0   \ar[r]_-{f_0} & B_0.
}
\]
\end{lemma} 

\begin{proof}
Using in-line notation $[X,Y] = Y^X$, take $$G = [{A_0}, B_0]\times_{[{A_1},B_0]} [{A_1}, B_1]$$ where the pullback  is formed with respect to $p$ and $q$, as in \eqref{diag:intLLP}.
\begin{equation}\label{diag:required}
\xymatrix{
[{A_0}, B_0]\times_{[{A_1},B_0]} [{A_1}, B_1] \ar[d] \ar[r] & [{A_1}, B_1] \ar[d]^{[{A_1}, q]} \\
[{A_0}, B_0] \ar[r]_{[p, B_0]}  	& 	[{A_1},B_0]
}
\end{equation}
The maps $e_i$ for $i=1,2$ are defined by $e_i\ =\ \ev_i \circ (p_i\times A_i)$:
\[
\xymatrix{
\big( [{A_0}, B_0]\times_{[{A_1},B_0]} [{A_1}, B_1]\big) \times A_i \ar[d]_{p_i\times A_i} \ar@/^4ex /[rrd]^-{e_i}&&  \\
 [{A_i}, B_i] \times A_i 	\ar[rr]_{\ev_i} && 	B_i
}
\]
We have $e_0 \circ (G\times p) = (\ev_0 \circ (p_0\times A_0)) \circ (G\times p) 
= q\circ (\ev_1 \circ (p_1\times A_1))$
Verification of the construction is left to the reader.
\end{proof}

%
%

Returning to the proof of the proposition, we first restore the products on the left in diagram \eqref{diag:pause} by restoring the indexing over $A_0$ and moving $q : B_1\to B_0$ to $\C/A_0$ by base change along $A_0\to 1$ (but without explicitly writing $A_0^*(B_1)$, etc.).   We now want to apply the lemma to the case of the category $\C/A_0$, with $q: B_1\to B_0$ as named in the lemma and $r_A : A_1 \to I_A$ in place of $p: A_1 \to A_0$.  Although $\C/A_0$ is not locally cartesian closed, the objects $B_0$ and $B_1$ and the maps $p:A_1 \to A_0$ and $I_A\to A_0$ are all in \D, and so the required exponentials exists in $\D(A_0)$, and thus in $\C/A_0$.  Moreover, the required pullback \eqref{diag:required} exists because $q$ is in \D.

Applying the lemma to the filling problem in diagram \eqref{diag:pause}, we can therefore interpolate the universal case $(G, e_0, e_1)$ to obtain the following (where we have written $\times$ for $\times_{A_0}$):
\[
\xymatrix{
X\times A_1 \ar@/^4ex /[rrr]^{c} \ar[d]_{X\times r_A} \ar[rr]_{f\times A_1} &
					& G\times A_1 \ar[d]_{G\times r_A} \ar[r]_-{e_1} & B_1 \ar[d]^{q} \\
X\times I_A   \ar@/_4ex /[rrr]_{b} \ar[rr]^{f\times I_A} &
					& G\times I_A   \ar[r]^-{e_0} & B_0
}
\]
where $f : X \to G$ classifies $(X,b,c)$.

Now $G\times r_A$ is anodyne, since $r_A$ is, so we can find a diagonal filler 
$j(e_1,e_0): G\times A_0 \to B_1$ for  this generic case.  
\[
\xymatrix{
X\times A_1 \ar@/^4ex /[rrrr]^{c} \ar[d]_{X\times r_A} \ar[rr]_{f\times A_1} &
					& G\times A_1 \ar[d]_{G\times r_A} \ar[rr]_-{e_1} && B_1 \ar[d]^{q} \\
X\times I_A   \ar@/_4ex /[rrrr]_{b} \ar[rr]^{f\times I_A} &
					& G\times I_A   \ar[rr]^-{e_0} \ar@{.>}[rru]|-{j(e_1,e_0)} && B_0
}
\] 
Then for any  lifting problem of the form $(X, b, c)$ in \eqref{diag:pause}, we can  take as a filler $j(c,b) = j(e_1,e_0)\circ (f\times I_A)$ to have a  choice that is natural in $X$.  This provides the required left-lifting structure for $(\delta, \iy)$ with respect to $\pi$.
\end{proof}

\subsection{The main result}

Combining propositions \ref{prop:sumprod} and \ref{prop:id}, we have now reached our goal:

\begin{theorem}\label{thm:natmod}
Let  $\D$ be any closed,  stable, factorizing class of maps in a category \C.  There is  a representable natural transformation $\pi:\D_1\to\D_0$ over $\C$ that models dependent type theory with extensional sums $\Sigma$, extensional products $\Pi$, and intensional identity types $\Id$.
\end{theorem}

\begin{corollary}\label{cor:CwF}
Let  $\D$ be any closed,  stable, factorizing class of maps in a category \C.  There is a category-with-families model of dependent type theory, with extensional sums $\Sigma$, extensional products $\Pi$, and intensional identity types $\Id$, with the contexts and substitutions being the objects and morphisms of \C, and as types and terms in context $X$, a category equivalent to the \D-maps into $X$ and their sections.
\end{corollary}

\begin{remark}\label{rem:credit}
A result essentially the same as our corollary \ref{cor:CwF} was announced in 2012 by Lumsdaine and Warren, and has finally appeared in \cite{LW}.  Reasoning very similar to that used here is also used in that work, which should therefore be regarded as prior.  The main contribution of the present work is the concept of a natural model of type theory as an alternative presentation of the notion of a category with families, and the adaptation of the results and methods of \cite{LW} to this setting.
\end{remark}

\noindent Examples of categories satisfying the conditions of  theorem \ref{thm:natmod} include:
\begin{enumerate}
\item the category of Kan simplicial sets, with the (right) weak factorization system of the associated Quillen model structure.
\item similarly, the category of fibrant objects in any locally cartesian closed model category that is right proper, and in which the cofibrations are the monos; e.g.\ any right proper, Cisinski model category.
\item more generally, the category of ``fibrant'' objects in any weak factorization system on a (pre)sheaf topos in which the left maps are preserved by pullback along the right maps (the ``Frobenius condition" of van den Berg and Garner \cite{GvdB}).
\item non-LCC examples of categories with a weak factorization system for which the right maps are exponentiable, such as groupoids and categories with iso-fibrations.
\item any $\pi$h-tribe, in the sense of Joyal's categorical axiomatics for homotopy type theory \cite{J}.
\item the syntactic category of contexts $\mathcal{C}(\mathbb{T})$ of a system of type theory $\mathbb{T}$ with $\Sigma, \Pi$ and $\Id$ types (see \cite{GG}).
\end{enumerate}

\begin{remark}
Regarding terminology:
Let  $\D$ be any closed,  stable, factorizing class of maps in a category \C.  We may call the maps in \D\  \emph{typical} (since they are the \emph{types}), and say that \D\ is a \emph{typical structure} on \C, and that \C\ (together with \D) is a \emph{typical category}.  Our main theorem then says that any typical category supports a natural model of basic homotopy type theory. 

Assuming a  class of maps \D\  that is stable and closed, it is enough to require anodyne-\D\ factorizations just for the diagonal maps $A\to A\times A$,  in order to obtain them for all maps.  The notion of a typical category is then closely akin to \emph{first-order logic}: a category of contexts and substitutions, equipped with a system of ``predicates" closed under $\Sigma$, $\Pi$, and $\Id$.  A notion of category suitable to model full homotopy type theory, with a (univalent) universe and higher inductive types, will then be a typical category with some additional structure.
\end{remark}

\subsubsection*{Acknowledgements}

The  results developed here are an amalgamation of original ideas and ones derived from \cite{KLV} and \cite{LW}.  The author has benefitted from conversations with Thierry Coquand, Nicola Gambino, Richard Garner, Andr\'e Joyal, Peter Lumsdaine, Andy Pitts, Michael Shulman, Thomas Streicher, Michael Warren, and Vladimir Voevodsky.  Thanks are also due to two anonymous referees, who contributed many insightful suggestions for improvement, and to Marco Larrea for a good late catch.  The author thanks the Institute for Advanced Study, where this research was mainly conducted and first presented, and the Institut Henri Poincar\'e, where it was concluded.  Support was provided by the Air Force Office of Scientific Research through MURI grant FA9550-15-1-0053, and by the National Science Foundation. Any opinions, findings, and conclusions or recommendations expressed in this material are those of the author and do not necessarily reflect the views of the AFOSR or the NSF.  


\end{document}